\documentclass[12pt]{amsart}

 \usepackage[margin=1in]{geometry}

\newcommand{\private}[1]{}
\usepackage{amssymb, amsmath, amscd, amsthm, color, epsfig,url, graphicx, url}

\usepackage[all]{xy}          % for xy-pic pictures
\xyoption{dvips}              % for xy-pic pictures

\makeatletter
\renewcommand\l@subsection{\@tocline{2}{0pt}{2pc}{5pc}{}}
\makeatother

\newcommand{\R}{{\mathbb R}}
\newcommand{\Z}{{\mathbb Z}}

\newcommand{\hofiber}{\operatorname{hofiber}}

\newcommand{\holim}{\operatorname{holim}}
\newcommand{\colim}{\operatorname{colim}}

\newcommand{\tfiber}{\operatorname{tfiber}}
\newcommand{\tkernel}{\operatorname{tkernel}}
\newcommand{\tcokernel}{\operatorname{tcokernel}}

\newcommand{\Emb}{\operatorname{Emb}}
\newcommand{\Ebar}{\overline{\Emb}}

\newcommand{\Imm}{\operatorname{Imm}}

\newcommand{\rImm}{\operatorname{rImm}}

\newcommand{\rbar}{\overline{\rImm}}

\newcommand{\Ho}{\operatorname{H}}

\newcommand{\coker}{\operatorname{coker}}

\newcommand{\Conf}{\operatorname{Conf}}
\newcommand{\rConf}{\operatorname{rConf}}
\newcommand{\codim}{\operatorname{codim}}

\newcommand{\HZ}{\underline{H}\Z}

\newcommand{\calC}{{\mathcal{C}}}
\newcommand{\calO}{{\mathcal{O}}}
\newcommand{\calP}{{\mathcal{P}}}
\newcommand{\calQ}{{\mathcal{Q}}}

\newcommand{\calA}{{\mathcal{A}}}

\newcommand{\Top}{\operatorname{Top}}

\theoremstyle{plain}
\newtheorem{thm}{Theorem}[section]
\newtheorem{prop}[thm]{Proposition}
\newtheorem{lemma}[thm]{Lemma}
\newtheorem{cor}[thm]{Corollary}

\theoremstyle{definition}
\newtheorem{defin}[thm]{Definition}
\newtheorem{example}[thm]{Example}
\newtheorem{def/ex}[thm]{Definition/Example}

\theoremstyle{remark}
\newtheorem{rem}[thm]{Remark}
\newtheorem{rems}[thm]{Remarks}

\usepackage{tikz}

\makeatletter
\@namedef{subjclassname@2020}{%
  \textup{2020} Mathematics Subject Classification}
\makeatother

\usepackage{graphicx}

\begin{document}
\pagestyle{plain}
\title{Intrinsic convergence of the homological Taylor tower for $r$-immersions in $\mathbb R^n$}

\author{Gregory Arone}
\address{Gregory Arone\newline Department of Mathematics, Stockholm University}
\email{gregory.arone@math.su.se}
%\urladdr{???}

\author{Franjo \v Sar\v cevi\'c}
\address{Franjo \v Sar\v cevi\'c \newline Department of Mathematics, University of Sarajevo}
\email{franjo.sarcevic@pmf.unsa.ba}
\urladdr{pmf.unsa.ba/franjos}

\subjclass[2020]{Primary: 18F50; Secondary: 57R42, 57R40, 55R80, 55P42}
\keywords{calculus of functors, manifold calculus, Taylor tower, embeddings, immersions, $r$-immersions, homotopy of spectra, homological convergence, partial configuration space}
\thanks{ F. \v Sar\v cevi\'c was partially supported by the grant P20\_01109 (JUNTA/FEDER, UE). He is also grateful to Ismar Volić for all his support.\\The final version of this paper is published in Homology, Homotopy and Applications, Volume 26 (2024) Number 2, pp. 163-192. https://dx.doi.org/10.4310/HHA.2024.v26.n2.a8} 
%This paper was written for the most part as one of the Ph.D. projects of the second author. The author would like to thank his Ph.D. supervisior Ismar Voli\'c for all his support. He would also like to thank the University of Seville where this paper was finally finished.}

\begin{abstract}
For an integer $r\ge 2$, the space of $r$-immersions of $M$ in $\R^n$ is defined to be the space of immersions of $M$ in $\R^n$ such that at most $r-1$ points of $M$ are mapped to the same point in $\R^n$. The space of $r$-immersions lies ``between" the embeddings and the immersions. We calculate the connectivity of the layers in the homological Taylor tower for the space of $r$-immersions in $\mathbb R^n$ (modulo immersions), and give conditions that guarantee that the connectivity of the maps in the tower approaches infinity as one goes up the tower. We also compare the homological tower with the homotopical tower, and show that up to degree $2r-1$ there is a ``Hurewicz isomorphism" between the first non-trivial homotopy groups of the layers of the two towers.
\end{abstract}

\maketitle

%%%%%%%%%%%%%%%%%%%%%%%%%%%%%%%%%%%%%%%%%%%%%%
\tableofcontents

%%%%%%%%%%%%%%%%%%%%%%%%%%%%%%%%%%%%%%%%%%%%%%%%%%%%%%%%%%%%%%%%%%%%%%%%%%%%%%%%%%%%%%%%%%%%%%%%%%%%%%%%

\parskip=5pt
\parindent=0cm

%%%%%%%%%%%%%%%%%%%%%%%%%%%%%%%%%%%%%%%%%%%%%%%%%%%%%%%%%%%%%%%%%%%%%%%%%%%%%%%%%

%\section{Introduction}\label{S:Intro}

%%%%%%%%%%%%%%%%%%%%%%%%%%%%%%%%%%%%%%%%%%%%%%%%%%%%%%%%%%%%%%%%%%%%%%%%%%%%%%%%%

\begin{section}{Introduction}
Let $M$ be a smooth manifold of dimension $m$, and fix an integer $r\ge 2$.  An \textit{$r$-immersion} of $M$ in $\R^n$ is an immersion of $M$ in $\R^n$ such that the preimage of every point in $\R^n$ contains at most $r-1$ points of $M$. The space of $r$-immersions of $M$ in $\R^n$ is denoted by $\rImm(M, \R^n)$. For $r=2$, $2$-immersions are the same thing as injective immersions, which are essentially the same as embeddings in nice cases. In any case, we have inclusions of subspaces
\[
\Emb(M, \R^n)\subseteq 2\Imm(M, \R^n)\subset 3\Imm(M, \R^n)\subset \cdots\subset \rImm(M, \R^n)\subset \cdots \subset \Imm(M, \R^n).
\]
In this paper we study the {\it homological Taylor tower} of the $r$-immersions functor. The ``Taylor tower'' is meant in the sense of manifold calculus (also known as embedding calculus) developed by Weiss \cite{W:EI1} and Goodwillie-Weiss \cite{GW:EI2}.

The basic idea of manifold calculus is the following. In order to study the homotopy type of a space such as $\rImm(M, \R^n)$, one views it as a particular value of the presheaf $\rImm( -, \R^n)$ defined on $M$ (one can also consider more general target manifolds than $\R^n$, but we will content ourselves with maps into $\R^n$). A presheaf is a contravariant functor on the poset $\mathcal{O}(M)$ of open subsets of $M$. Inside $\mathcal{O}(M)$ there is a sequence of subposets $\mathcal{O}_1(M)\subset \cdots \mathcal{O}_k(M)\subset \cdots \subset \mathcal{O}_\infty(M)$, where $\mathcal{O}_k(M)$ is the poset of open subsets of $M$ that are diffeomorphic to the disjoint union of at most $k$ copies of $\R^m$. By restricting a presheaf $F$ to $\mathcal{O}_k(M)$ and then extrapolating back to $\mathcal{O}(M)$ one obtains a tower of approximations to $F$, which is usually denoted as follows
\[
F\to (T_\infty F \to \cdots \to T_kF \to T_{k-1}F \to \cdots T_0F).
\]
This is called the ``Taylor tower'' of $F$. Manifold calculus, and the Taylor tower in particular, has had many consequences and applications \cite{Mun05}, \cite{Vol06}, \cite{ALV07}, \cite{Mun11}, \cite{DwyerHess:LongKnots}, \cite{ST16}, \cite{BdBW18}. 

In this paper we investigate the Taylor tower that calculates the {\it homology} of the space $\rImm(M, \R^n)$. In practice, this means the following. First of all, it is convenient to replace the space of $r$-immersions with $r$-immersions modulo immersions. Let us suppose that we fix a basepoint in $\Imm(M, \R^n)$, and let $\rbar(M, \R^n)$ be the homotopy fiber of the inclusion map $\rImm(M, \R^n)\to \Imm(M, \R^n)$. Let $\HZ$ denote the Eilenberg-MacLane spectrum. We are interested in the Taylor tower of the presheaf of Spectra, defined by the formula
\[
U \mapsto  \HZ\wedge  \rbar(U, \R^n).
\] 
(more precise definitions are given in Section~\ref{S:Prereq}). 

Our main result concerns the rate of convergence of the Taylor tower of this functor. The question of convergence is a fundamental one. We will distinguish between two aspects of convergence: how rapidly the tower converges to its limit, and what it converges to.  
We will say that the Taylor tower of a functor $F$ 
\begin{enumerate}
\item {\it converges} at $M$ if the map $F(M)\to \holim_k T_kF(M)$ is an equivalence.
\item {\it converges intrinsically} at $M$ if the connectivity of the map $T_kF(M)\to T_{k-1}F(M)$ approaches $\infty$ as $k$ approaches $\infty$. 
\item {\it converges strongly} at $M$ if the connectivity of the map $F(M)\to T_kF(M)$ approaches $\infty$ as $k$ approaches $\infty$. 
\end{enumerate}
It is clear that (1)+(2)$\iff$(3). In particular, strong convergence implies intrinsic convergence, but the converse does not have to be true. In practice it seems that for ``natural'' functors that we know, whenever the Taylor tower of $F$ converges intrinsically, it converges strongly to $F$. But intrinsic convergence is usually much easier to prove than strong convergence.

Before we state our main result, let us recall, for context, that one of the deepest results in functor calculus is the Goodwillie-Klein-Weiss convergence theorem \cite{GW:EI2}, \cite{GK:EmbDisj}, \cite{GK}.

\begin{thm}[Convergence of the Taylor tower for spaces of embeddings]\label{T:GKW-Convergence}
%The functor $$U \longmapsto \Emb(U,N)$$ is $(n-2)$-analytic with excess $3-n$.\\
%Hence, if 
If $M$ is a smooth closed manifold of dimension $m$, and $N$ is a smooth manifold of dimension $n$, then the map $$\Emb(M,N) \to T_k\Emb(M,N)$$ is $$k(n-m-2)+1-m\text{-connected}.$$ In particular, if $n-m-2>0$, then the connectivities grow with $k$ and the Taylor tower therefore converges strongly to $\Emb(M,N)$.
\end{thm}

%Compare Theorem \ref{T:GKW-Convergence} with Theorem \ref{T:GW99-F-TkF}.

There is an easier, but also important convergence result for the homological version of the tower, which is more directly relevant to this paper. Define $\Ebar(M, \R^n)$ to be the homotopy fiber of the inclusion $\Emb(M, \R^n)\to \Imm(M, \R^n)$. Consider the contravariant functor from $\mathcal{O}(M)$ to Spectra that sends $U$ to $\underline{H}\mathbb Z\wedge \Ebar(U, \R^n)$. This functor represents the homology of the space of embeddings modulo immersions. The Taylor tower of this functor is known to converge when $n > 2m+1$~\cite{W:HomEmb}.

Now let us state our main result
\begin{thm}\label{thm: convergence} Let $M$ be $m$-dimensional. Assume that $n\ge 2$.
If $r\le n+1$, the Taylor tower for $\HZ \wedge \rbar(M, \R^n)$ converges intrinsically when
\[
n > \frac{rm+1}{r-1}.
\]
If $r\ge n+1$ then the Taylor tower converges intrinsically when $n> m+1$.
\end{thm}
\begin{rems}\hspace{10cm}
\begin{enumerate}
\item When $r=n+1$ the two statements are equivalent, by an easy calculation. %Indeed, the function $f(n)=n^2-nm-m-1, n\in \mathbb N$, is positive only for $n>m+1$.
\item When $r=2$ we get the condition $n>2m+1$, which is the known condition for the convergence of the Taylor tower of $\HZ\wedge \Ebar(M, \R^n)$.
\item The condition $n > \frac{rm+1}{r-1}$ is equivalent to $rm - (r-1)n < -1$. The number $rm - (r-1)n$ equals, at least when it is positive, to the dimension of the intersection of $r$ copies of $\R^m$ embedded in $\R^n$ in a general position.
\end{enumerate}
\end{rems}

Next let us discuss the proof. Let $F$ be a presheaf defined on a suitable category of $m$-dimensional manifolds and codimension zero embeddings. The basic building blocks in the construction of the Taylor tower of $F$ are spaces of the form $F(\coprod_i \R^m)$, for $i=0,1,2,\ldots$. The homotopy fiber of the map $T_kF \to T_{k-1}F$ depends on the total homotopy fiber of the following cubical diagram, indexed by the poset of subsets of $\underline{k}=\{1, \ldots, k\}$:
\begin{equation}\label{eq:cube}
S\mapsto F\left(\coprod_{\underline{k}\setminus S} \R^m\right)
\end{equation}
This homotopy fiber is sometimes called the $k$-th derivative (or the $k$-th cross-effect) of $F$ at $\emptyset$. The following fact is particularly important for analysing intrinsic convergence. Recall that a cubical diagram is called {\it $c$-cartesian} if the map from the initial object to the homotopy limit of the rest of the cubical diagram is $c$-connected. Suppose the cubical diagram~\eqref{eq:cube} is $c_k$-cartesian. Then the map $T_kF(M) \to T_{k-1}F(M)$ is $c_k-mk$-connected. Thus the Taylor tower of $F$ converges intrinsically at $M$ if the number $c_k-mk$ approaches $\infty$ as $k$ approaches $\infty$.

When $F(M)=\Ebar(M, \R^n)$, there is a well-known equivalence $\Ebar(\coprod_k \R^m, \R^n)\simeq \Conf(k,\R^n)$, where $\Conf(k,\R^n)$ is the configuration space of ordered $k$-tuples of pairwise distinct points in $\R^n$.
%$$\Conf(k,\R^n):=\Emb(\underline{k},\R^n)\cong \{(x_1,x_2,...,x_k)\in {(\R^{n}})^k: x_i \neq x_j \text{ for } i \neq j\},$$
%where $\underline{k}=\{1,2,...,k\}$. 
Similarly, there is an equivalence between $\rbar(\coprod_k \R^m, \R^n)$ and the so-called $r$-configuration space, also called \textit{no $r$-equal} configuration space, defined by $$\rConf(k,\R^n):= \rImm(\underline{k},\R^n).$$ 
This is the space of ordered $k$-tuples of points in $\R^n$ where at most $r-1$ are allowed to be equal. A proof of the equivalence 
\[
\rbar(\coprod_k \R^m, \R^n)\xrightarrow{\simeq}\rConf(k, \R^n)
\]
is given in \cite{AS:rimmrconf}. Thus $r$-configuration spaces are basic building blocks in the Taylor tower of $\rImm(M, \R^n)$.

%There is an important difference between classic configuration spaces and $r$-configuration spaces. Namely, while the projection maps $$\Conf(j,N) \rightarrow \Conf(j-1,N)$$ that forget a point are fibrations \cite{FN:ConfFibration}, when $r>2$ the projection  maps
%\begin{equation*}\label{E:ProjectionrConf}
%\rConf(j,N) \rightarrow \rConf(j-1,N)
%\end{equation*}
%are not fibrations (see \cite[Example 3.2]{SSV:r-imm}). It follows that the total homotopy fiber of the cubical diagram of $r$-configuration spaces $S\mapsto \rConf(\underline{k}\setminus S, \R^n)$ is harder to analyse for $r>2$ than in the case of classical configuration spaces. On the other hand, $r$-configuration spaces are complements of subspace arrangements, and their homology is accessible by means of the Goresky-MacPherson formula and other such tools. This was done by Bjorner and Welker~\cite{Bjo95} and others.

To analyse the intrinsic convergence of the Taylor tower of the functor $\HZ\wedge \rbar(-, \R^n)$, one needs to calculate how cartesian the following $k$-dimensional cubical diagram is
\begin{equation}\label{eq: homology cube}
S\mapsto \HZ \wedge \rConf(\underline{k}\setminus S, \R^n).
\end{equation}
The space $\rConf(i, \R^n)$ is the complement of a subspace arrangement in $\R^{ni}$. It follows that the homology of $r$-configuration spaces is accessible by means of the Goresky-MacPherson formula and other such tools. The homology of $r$-configuration spaces was studied by a number of people, starting with Bj\"orner and Welker~\cite{Bjo95}.

Using the Goresky-MacPherson formula and the results in~\cite{Bjo95} we prove the following result (it is combining Proposition~\ref{prop: conf connectivity} and Theorem~\ref{T:Convergence})
\begin{thm}\label{thm: connectivity}
When $r\le n+1$, the cube~\eqref{eq: homology cube} is $k(n-1)+\left\lfloor \frac{k}{r}\right\rfloor (r-n-1)$-cartesian, and the map
\begin{equation*}
p_k\colon T_k\HZ\wedge \rbar(M, \R^n)\to T_{k-1}\HZ\wedge \rbar(M, \R^n)
\end{equation*}
is $$k\left(n\frac{r-1}{r}-m-\frac{1}{r}\right)-\frac{(k\hspace{-8pt}\mod r)}{r} (r-n-1)\mbox{-connected.}$$ Here $(k\hspace{-4pt}\mod r):=k-r\left\lfloor \frac{k}{r}\right\rfloor$.

When $r\ge n+1$, the cube~\eqref{eq: homology cube} is $k(n-1)+r-n-1$-cartesian, and the map $p_k$ is $$k(n-m-1)+r-n-1\mbox{-connected.}$$
\end{thm}
Theorem~\ref{thm: convergence} follows easily from Theorem~\ref{thm: connectivity}.

In Section~\ref{section:compare} we compare the tower of the homological functor $\HZ\wedge \rbar(M, \R^n)$ with that of the tower of the homotopical functor $\rImm(M, \R^n)$. Let us suppose that we chose a basepoint in the space $\rImm(M, \R^n)$. In this case the presheaf $\rbar(-, \R^n)$ takes values in pointed spaces, and we have the following diagram of presheaves:
\begin{equation}\label{eq: comparison}
\rImm(-, \R^n)\xleftarrow{i} \rbar(-, \R^n) %\xrightarrow{s}  \Omega^\infty \Sigma^\infty \rbar(-, \R^n) 
\xrightarrow{h} \Omega^\infty\HZ\wedge \rbar(-, \R^n).
\end{equation}
It is well-known that the map $i$ induces an equivalence of all layers except the first one. Indeed, the map $i$ is the homotopy fiber of the map from $\rImm(-, \R^n)$ to its linear approximation. Thus we can view the map $h$ as a map from the higher layers/derivatives of $\rImm(-, \R^n)$ to the corresponding layers/derivatives of $\Omega^\infty\HZ\wedge \rbar(-, \R^n)$, which are essentially the same as the layers/derivatives of $\HZ\wedge \rbar(-, \R^n)$, since $\Omega^\infty$ commutes with Taylor approximations.

When $r=2$, the second derivative of $\rImm(-, \R^n)$ is equivalent to $S^{n-1}$, and the second derivative of $\HZ\wedge \rImm(-, \R^n)$ is $\HZ\wedge S^{n-1}$. It follows that in the case $r=2$, the map $h$ in~\eqref{eq: comparison} induces the Hurewicz homomorphism from the second derivatives of $\rImm(-, \R^n)$ to the second derivative of $\HZ\wedge \rbar(-, \R^n)$. In particular, it follows that the connectivity of the quadratic layers of the Taylor towers of $\rImm(-, \R^n)$ and of $\HZ\wedge \rbar(-, \R^n)$ is the same, and their first non-trivial homotopy groups are isomorphic. 

By contrast, at degrees higher than $2$, the layers of the homotopical tower $\rImm(-, \R^n)$ and of the homological tower of the functor $\HZ\wedge \rbar(-, \R^n)$ have different connectivities, and there is no Hurewicz type isomorphism between them.

And again by contrast, in Section~\ref{section:compare} we show that for $r>2$ the map $h$ in diagram~\eqref{eq: comparison} induces a Hurewicz type isomorphism between first non-trivial homotopy groups of layers roughly up to degree $2r-1$. See Theorem~\ref{thm: comparison} for precise statement.%\GregNote{Maybe it is worthwhile to quote Theorem~\ref{thm: comparison} in the introduction}

\subsubsection*{Organization of the paper}
In Section~\ref{S:Prereq} we review some background material on cubical diagrams, manifold calculus and spectra. In Section~\ref{S:main} we introduce the homological Taylor tower that is the main subject of this paper. 

In Section \ref{S:GM} we make an excursion into the subspace arrangements. We describe $r$-configuration spaces via subspace arrangements and compute their cohomology using the Goresky-MacPherson theorem. 

In Section~\ref{S:Retractive} we define the notion of a retractive cubical diagram. This is a diagram where the maps have sections that satisfy a certain hypothesis. We prove that the homotopy groups of the total homotopy fiber of a retractive cube are isomorphic to the total kernel of the cube of homotopy groups.

In Section~\ref{S: Conf is retractive} we prove that the cube of $r$-configuration spaces that controls the layers in the Taylor tower is retractive. In Section~\ref{S:ccrcs} we prove our main result about the homological connectivity of the cube of $r$-configuration spaces. In Section~\ref{convergenceresult} we prove the main result about the intrinsic convergence of the Taylor tower of $\HZ\wedge \rbar(M, \R^n)$.

In Section~\ref{section:compare} we compare the tower of $\HZ\wedge \rbar(M, \R^n)$ with the tower of $\rImm(M, \R^n)$ in low degrees. We prove that the layers in the two towers have the same connectivity up to degree $2r-1$ (with some exceptions in the cases $r=2, 3$).

In Section~\ref{S:Further} we discuss some possible directions for further exploration.

\end{section}

\vspace*{0.5cm}
\begin{section}{Prerequisites}\label{S:Prereq}
\begin{subsection}{Cubical diagrams}\label{SS:cubicaldiagrams}
Cubical diagrams play an important role in functor calculus, and in this paper in particular, so we will recall a few elementary facts about them. All the results in this subsection, and much more, can be found in~\cite{CalcII}.

Let $\underline{k}$ denote the standard set with $k$ elements $\{1, \ldots, k\}$. Let $\calP(\underline{k})$, or just $\calP(k)$, denote the poset of subsets of $\underline{k}$. A $k$-dimensional cubical diagram in a category $\calC$ is a functor $\chi\colon \calP\to \calC$. It is easy to see that $\calP(k)$ is equivalent to $\calP(k)^{\mathrm{op}}$, so a contravariant functor from $\calP(k)$ to $\calC$ is called a cubical diagram as well.
We will mostly consider cubical diagrams in (pointed) spaces and spectra, and also in abelian groups.

Given a cubical diagram $\chi$ in topological spaces or spectra, there is a natural map
\[
i_\chi\colon \chi(\emptyset)\to \underset{\emptyset \ne S\subset \underline{k}}{\holim} \ \chi(S).
\]
We say that $\chi$ is $c$-cartesian, if this map is $c$-connected. The homotopy fiber of this map is called the {\it total homotopy fiber} of $\chi$. The total homotopy fiber of $\chi$ is denoted by $\tfiber(\chi)$. Clearly if $\chi$ is $c$-cartesian then $\tfiber(\chi)$ is $c-1$-connected. The converse always holds for cubical diagrams of spectra, and it holds for spaces under the additional assumption that $i_\chi$ is surjective on path components.

One can identify a $k$-dimensional cubical diagram with a map of two $k-1$-dimensional cubical diagrams. Given a $k$-dimensional cubical diagram $\chi$, let us define two $k-1$-dimensional cubical diagrams $\chi_1$ and $\chi_2$ as follows: $\chi_1(U)=\chi(U)$, and $\chi_2(U)=\chi(U\cup \{k\})$. Then $\chi$ can be identified with the map of cubes $\chi_1\to \chi_2$. Furthermore, there is a homotopy fibration sequence whose meaning is that total homotopy fiber can be calculated as an iterated homotopy fiber
\[
\tfiber(\chi)\simeq \hofiber(\tfiber(\chi_1)\to \tfiber(\chi_2)).
\]

When $\chi$ is a cubical diagram of abelian groups, we define the total kernel of $\chi$ to be 
\[
\tkernel(\chi):=\ker(\chi(\emptyset) \to \prod_{i=1}^k \chi(\{i\}))
.\] 
Just as with total fibers, the total kernel can be calculated as an iterated kernel. There is a natural isomorphism
\[
\tkernel(\chi)\cong \ker(\tkernel(\chi_1)\to \tkernel(\chi_2)).
\]
When $\chi$ is a cubical diagram of spaces or spectra, there is a natural homomorphism of graded groups
\[
\pi_*(\tfiber\chi)\to \tkernel(\pi_*\chi).
\]
This homomorphism is not an isomorphism in general. In Section~\ref{S:Retractive} we will investigate a condition on a cubical diagram that guarantees for it to be an isomorphism.
\end{subsection}
\begin{subsection}{Manifold calculus of functors}\label{SS:manifoldcalculus}

Let $M$ be a smooth manifold of dimension $m$. Define $\mathcal{O}(M)$ to be the poset category of open subsets of $M$. Objects of $\mathcal{O}(M)$ are open sets $U \subseteq M$, and morphisms $U \rightarrow V$ are the inclusions $U \subseteq V$.

Manifold calculus of functors, developed by Weiss \cite{W:EI1} and Goodwillie-Weiss \cite{GW:EI2}, studies contravariant functors from $\mathcal{O}(M)$ to a category that supports a reasonable notion of homotopy. In their foundational papers, Goodwillie and Weiss only considered functors with values in topological spaces, and maybe spectra. Nowadays it is natural to let the target category to be an $\infty$-category. We will content ourselves with functors with values in (pointed) spaces and in spectra. 

Technically speaking, manifold calculus applies to functors that are {\it good}, in the sense that they satisfy the following two conditions:\\
\hspace*{1cm} (i) they are \textit{isotopy} functors, and\\
\hspace*{1cm} (ii) they are \textit{finitary}.\\
A functor is an isotopy functor if it takes isotopy equivalences to weak homotopy equivalences (for the definition of \textit{isotopy equivalence} see \cite[Definition 10.2.2]{MV:Cubes}). It is finitary if for every monotone union $\bigcup_{i} U_i$ (where $U_i \subset U_{i+1}$ for $i=1,2,...$) the canonical map from $F(\bigcup_{i} U_i)$ to $\holim_i F(U_i)$ is a weak homotopy equivalence. %Here $\holim$ denotes \textit{homotopy limit}.

If $F$ is a "half-good" contravariant functor (cofunctor), i.e. an isotopy functor which is not a finitary functor, then we need to \text{tame} this functor. We call $V \in \mathcal{O}(M)$ \textit{tame} if $V$ is the interior of a compact smooth codimension zero submanifold of $M$. As mentioned in [GKW01], property (ii) ensures that a good cofunctor $F$ on $\mathcal{O}(M)$ is essentially determined by its behavior on tame open subsets of $M$.

In particular, suppose $F$ is a cofunctor from $\mathcal{O}(M)$ to $\Top$ having property (i). Then the functor defined by $$F^{\#}(V):=\holim_{\text{tame } U \subset V} F(U)$$ for $V \in \mathcal{O}(M)$ has also property (ii), i.e. $F^{\#}$ is a good cofunctor on $\mathcal{O}(M)$. We call $F^{\#}$ the \textit{taming} of $F$.

There exists a natural transformation $F \to F^{\#}$. The map $F(V)\to F^{\#}(V)$ is an equivalence whenever either $F$ or $V$ is tame. 

The motivating example for the development of the manifold calculus of functors is the \textit{embedding} functor.
\begin{defin}(Space of embeddings)\label{D:embed} Let $M$ and $N$ be smooth manifolds.
\begin{itemize}
\item{A \textit{smooth embedding} of $M$ in $N$ is a smooth map $f:M \rightarrow N$ such that\\
1. the map of tangent spaces $$D_x f: T_x M \rightarrow T_{f(x)}N$$
\hspace*{0.47cm}is an injection for all $x \in M$, i.e. the derivative of $f$ is a fiberwise injection, and\\
2. $f: M \rightarrow f(M)$ is a homeomorphism.}\\
\item{The \textit{space of embeddings}, $\Emb(M,N)$, is the subspace of the space of smooth maps from $M$ to $N$ consisting of smooth embeddings of $M$ in $N$. The space $\Emb(M,N)$ is topologized using Whitney $\mathcal{C}^{\infty}$-topology; for an explanation see \cite[Appendix A.2.2]{MV:Cubes}).}
\end{itemize}
\end{defin}
%One of the most important examples of embeddings are \textit{configuration spaces} in manifolds.
%\begin{defin}(Configuration space)\\
%Let $N$ be a manifold and $\underline{k}=\{1,2,...,k\}$. The configuration space of $k$ points in $N$ is defined to be
%$$\Conf(k,N):=\Emb(\underline{k},N)\cong \{(x_1,x_2,...,x_k)\in N^k: x_i \neq x_j \text{ for } i \neq j\}.$$
%\end{defin}
An important example of a space of embeddings with very rich theory is the space of classical \textit{knots} defined to be the space $\Emb(S^1,\R^3)$. 
\begin{defin}(Embedding functor)\\
For a smooth $n$-dimensional manifold $N$, the \textit{embedding functor} $\Emb(-,N) : \mathcal{O}(M) \to \Top$ is a contravariant functor given by $U \mapsto \Emb(U,N)$.
\end{defin}
The contravariance follows from the fact that an inclusion of open subsets of a manifold $M$ gives a restriction map of embedding spaces of manifolds.

A related notion is the space of \textit{immersions} $\Imm(M,N)$, which is a space of smooth maps $f:M \rightarrow N$ such that just the derivative of $f$ is a fiberwise injection, (property 1. from Definition \ref{D:embed}). If $M$ is a compact manifold and $f$ is an injective immersion $M \rightarrow N$, then $f$ is an embedding.

The corresponding functor is the \textit{immersion functor} $\Imm(-,N): \mathcal{O}(M) \to \Top$ given by $U \mapsto \Imm(U,N)$.

Functors $\Emb(-,N)$ and $\Imm(-,N)$ are examples of good functors (see \cite{W:EI1} and \cite{GKW:EmbDisjSurg}).

The idea of the manifold calculus of functors is to approximate a good functor with simpler, polynomial functors.
\begin{defin}(Polynomial functor)\\
A good contravariant functor $F: \mathcal{O}(M) \to \Top$ is called \textit{polynomial of degree $\leq k$} if for all $U \in \mathcal{O}(M)$ and for all pairwise disjoint closed subsets $A_0, ..., A_k \subset U$, the $(k+1)$-cube $$\mathcal{P}(\underline{k+1}) \to \Top$$ $$S \mapsto F(U - \bigcup_{i\in S}A_i)$$ is homotopy cartesian; equivalently, the map $F(U) \to \holim_{S \neq \emptyset} F(U - \bigcup_{i\in S}A_i)$ is a homotopy equivalence. Here $\mathcal{P}(\underline{k+1})$ is the poset category of all subsets of the set $\underline{k+1}=\{1,...,k+1\}$ with $\subset$ as the relation of partial order. Its shape is an $(k+1)$-dimensional cubical diagram.
\end{defin}

%It is well known that a polynomial $f: \R \rightarrow \R$ of degree $k$ is uniquely determined by its values on $k+1$ distinct points; in other words, 
It is well known that a polynomial $f: \R \rightarrow \R$ of degree $k$ such that $f(0)=0$ is uniquely determined by its values on $k$ distinct points. In analogy, a polynomial functor is completely determined by its values on the category of at most $k$ open discs. \cite{M:MfldCalc} provides more analogies between the ordinary calculus of functions and the manifold calculus of functors.

More precisely, let $\mathcal{O}_k(M)$ be the full subcategory of $M$ consisting of open subsets of $M$ diffeomorphic to $\leq k$ disjoint discs. We have the following theorem due to Weiss (\cite[Theorem 5.1]{W:EI1}).
\begin{thm}\label{T:OkO}
Suppose $F,G: \mathcal{O}(M) \longrightarrow \Top$ are good functors that are polynomials of degree $\leq k$. If $T:F \to G$ is a natural transformation that is an equivalence for all $U \in \mathcal{O}_k(M)$, then $T$ is an equivalence for all $U \in \mathcal{O}(M)$.
\end{thm}
\begin{example}\label{E:PolynomialFunctors}\hspace{12cm}
\begin{itemize}
\item{The functor $U \mapsto \Imm(U,N)$ is a polynomial of degree $\leq 1$.}
\item{The functor $U \mapsto \Emb(U,N)$ is not a polynomial of degree $\leq k$ for any $k$.}
\end{itemize}
For the details, see \cite[Example 10.2.10]{MV:Cubes}, \cite[Example 2.3]{W:EI1}, \cite[Examples 4.7 and 4.8]{M:MfldCalc}.
\end{example}
\begin{defin}(Polynomial approximations)\\
For a good functor $F$, define for each $U \in \mathcal{O}(M)$ the \textit{$k^{\text{th}}$ polynomial approximation} of $F$ to be $$T_kF(U)=\holim_{V\in \mathcal{O}_k(U)}F(V).$$
\end{defin}
As Weiss proved in \cite[Theorems 3.9. and 6.1]{W:EI1}, such defined $T_kF$ is polynomial of degree $\leq k$. Also, higher derivatives of such defined polynomial functors vanish and derivatives of a functor and derivatives of its $k^{\text{th}}$ polynomial approximation agree up to $k^{\text{th}}$ degree, where the derivatives of functors are defined as follows:
\begin{defin}\label{D:Derivative}(Derivative of a functor)\\
Let $D_{1}^m, ..., D_{k}^m$ be pairwise disjoint open discs in $M$. Define a $k$-cube of spaces by the rule $S \mapsto F(\bigcup_{i \notin S}
D_{i}^m)$. We define the \textit{$k^{\text{th}}$ derivative} of $F$ at the empty set, denoted $F^{(k)}(\emptyset)$, to be the total homotopy fiber of the cube $S \mapsto F(\bigcup_{i \notin S}D_{i}^m)$.
\end{defin}
For example, the $1^{\text{st}}$ derivative of embeddings are immersions. Also, the linearization of the space of embeddings is the space of immersions, namely there exists an equivalence $T_1 \Emb(-,N) \simeq \Imm(-,N)$ (\cite{W:EI1}).

For more details and intuition behind this, see Munson's survey \cite{M:MfldCalc}. For other relevant results, see \cite[Theorem 10.2.16]{MV:Cubes} and \cite{W:EI1}.

%The homotopy limit in the definition of the $k^{\text{th}}$ stage is taken over a large category, but the advantage is in fact that it is functorial. There is an alternative definition of the $k^{\text{th}}$ stage expressed iteratively by a homotopy limit over a finite category (\cite[Example 10.2.18]{MV:Cubes}, \cite[pages 105 and 106]{SV:Convergence}, \cite{10.1007/978-3-030-24986-1_1}), but it loses the functorial property.

%Also there exists a natural functor transformation $F \longrightarrow T_kF$ induced by the inclusion $\mathcal{O}_k(U) \rightarrow \mathcal{O}(U)$, so it makes sense to discuss the "$k^{\text{th}}$ remainder" $R_kF=\hofiber(F \longrightarrow T_kF)$.

The inclusion $\mathcal{O}_{k-1}(U) \rightarrow \mathcal{O}_k(U)$ induces a map $T_kF(U) \to T_{k-1}F(U)$ and so we obtain a tower of functors, called the \textit{manifold calculus Taylor tower} of $F$:

\begin{equation}\label{E:Tower}
 %\resizebox{.9\hsize}{!}{
\xymatrix{
&   &     F(-)   \ar[dll]\ar[d]\ar[dr]\ar[drrr]          &                      &               &      \\
T_{0}F(-)  & \ar[l] \cdots &   T_{k-1}F(-) \ar[l] &    T_{k}F(-) \ar[l]   & \ar[l] \cdots &   T_{\infty}F(-) \ar[l]                          
}%}
\end{equation}
Here $T_{\infty}F$ denotes the \textit{homotopy inverse limit} of this tower. $T_k F$ is also called the \textit{$k^{th}$ stage} of the Tower.\\
By evaluating diagram \eqref{E:Tower} on $U \in \mathcal{O}(M)$, we get a diagram of spaces with maps between the stages that are fibrations. In particular, we can set $U=M$.
\begin{defin}(Layer)\\
Define the \textit{$k^{\text{th}}$ layer} of the manifold calculus Taylor tower of  $F$ to be the homotopy fiber of the map between two successive stages of the tower, that is, $$L_kF=\hofiber(T_kF \to T_{k-1}F).$$
\end{defin}
We need to work here with a based Taylor tower. It can be accomplished by choosing a basepoint in the space $F(M)$ which then also bases the spaces $T_kF(U)$ for all $k$ and $U$.

One of the fundamental results, which is a consequence of the Classification of \textit{homogeneous functors} theorem (\cite[Theorem 8.5]{W:EI1}, see also \cite[Theorem 10.2.23 and Proposition 10.2.26]{MV:Cubes}) is the following
\begin{prop}\label{P:ConnOfLayer}
For a good functor $F$ defined on $m$-dimensional manifolds, if 
%$F^{(k)}(\emptyset)$ 
the cube $$S\mapsto F\left(\coprod_{\underline{k} \setminus S}D^m\right)$$ is $c_k$-cartesian, then the map $T_kF(M)\to T_{k-1}F(M)$ is $c_k-km$-connected. 
%$L_kF(M)$ is $(c_k-km)$-connected. 
More generally, if $U$ has handle dimension $j$, then the map $T_kF(U)\to T_{k-1}F(U)$ is $(c_k-kj)$-connected.
\end{prop}
For the definition of \textit{handle dimension}, see \cite[Appendix A.2.1]{MV:Cubes}. 

It follows that the Taylor tower of $F$ converges intrinsically at $M$ if the number $c_k-mk$ approaches $\infty$ as $k$ approaches $\infty$.

\end{subsection}

\vspace*{0.5cm}
\begin{subsection}{Spectra}
The subject of this paper is a functor that represents homology. We had a choice between working with chain complexes and the singular chains functor, or working with spectra and using smash product with the Eilenberg-MacLane spectrum to represent homology. We chose the latter.

We adopt a naive, old-fashioned view of spectra as sequences of spaces equipped with structure maps between them.
\begin{defin}(Spectrum)\\
A spectrum $\underline{E}$ is a sequence of based spaces $\{E_n\}_{n\in \mathbb{N}_0}$ together with basepoint-preserving maps (called \textit{structure maps})
\begin{equation}\label{E:SigmaMap}
\Sigma E_n \rightarrow E_{n+1},
\end{equation}
or, equivalently, the maps
% \eqref{E:SigmaMap} is equivalent to the space of maps
\begin{equation}\label{E:MapOmega}
E_n \rightarrow \Omega E_{n+1},
\end{equation}
where $\Sigma$ and $\Omega$ denote suspension and loop space, respectively.

If the maps \eqref{E:MapOmega} are weak equivalences, then $\underline{E}$ is called an \textit{$\Omega$-spectrum}. Each $E_n$ from an $\Omega$-spectrum is called an \textit{infinite loop space}.
\end{defin}
%\begin{example}(Suspension spectrum)\label{Ex:SuspSpec}\\
%Let $E_n = \Sigma^n X$ for a based space $X$. The maps \eqref{E:SigmaMap} are then indentities and we have the suspension spectrum $\Sigma^{\infty} X$ of $X$.
%\end{example}
%A suspension spectrum is not necessarily an $\Omega$-spectrum. Setting $X=S^0$ we have the \textit{sphere spectrum}.

%The following result is important for the definition of Eilenberg-MacLane spectrum. The details can be found in \cite[Section 4.3]{Hatcher}.

%\begin{thm}\label{T:HomotopyOmega}
%Let $X$ be a based space. Then there exists an isomorphism $$\pi_{n+1}(X) \cong \pi_n(\Omega X).$$
%\end{thm}
\begin{example}(Eilenberg-MacLane spectrum)\\
Let $n$ be an arbitrary positive integer and $G$ be an arbitrary group, abelian for $n>1$. Then there exists a $\text{CW}$ complex $X$ such that
\begin{equation}\label{E:EMcLaneDef}
\pi_n(X) \cong G \text{ and } \pi_k(X) \text{ is trivial for } k \neq n.
\end{equation}
A topological space $X$ with property \eqref{E:EMcLaneDef} is called an \textit{Eilenberg-MacLane space} $K(G,n)$. For example, $K(\Z,1) \simeq S^1$.\\
For an abelian group $G$, the \textit{Eilenberg-MacLane spectrum}, denoted by $\underline{H}G$, is defined to be the spectrum $\{E_n\}_{n\in \mathbb{N}_0}$  with $E_n=K(G,n+1)$ and maps
\begin{equation}\label{E:MapKOmega}
K(G,n+1) \rightarrow \Omega K(G,n+2).
\end{equation}
The maps \eqref{E:MapKOmega} are weak equivalences, 
%by Theorem \ref{T:HomotopyOmega}
hence $\underline{H}G$ is an $\Omega$-spectrum.
\end{example}
%\begin{defin}(Wedge sum)\\
%Let $X$ and $Y$ be based spaces with basepoints $x_0$ and $y_0$ respectively. The \textit{wedge sum} $X \vee Y$ of $X$ and $Y$ is defined to be the quotient space  $$ X \vee Y = X \sqcup Y / \sim$$ where $\sim$ is the equivalence relation generated by $x_0 \sim y_0$.
%\end{defin}
%\begin{defin}(Smash product)\\
%Let $X$ and $Y$ be based spaces. The \textit{smash product} $X \wedge Y$ of $X$ and $Y$ is defined to be the quotient space $$ X \wedge Y = X \times Y / X \vee Y.$$
%\end{defin}
%\begin{example}\label{Ex:Smash}\hspace*{8cm}
%\begin{itemize}
%\item{If $X$ is a based space and $S^n$ is based sphere, then $S^n \wedge X \cong \Sigma^n X$. In particular, if $X=S^m$ then $S^n \wedge S^m \cong S^{m+1}$.}
%\item{$\Sigma(X \wedge Y) \cong (\Sigma X) \wedge Y \cong X \wedge (\Sigma Y)$.}
%\end{itemize}
%\end{example}

Since for a spectrum $\underline{E}$ there exist maps $$\pi_{i+n}(E_n) \rightarrow \pi_{i+n+1}(E_{n+1})$$ (for details, see \cite[Section 4.F]{Hatcher}), it makes sense to define the \textit{$i^{\text{th}}$ homotopy group of the spectrum} $\underline{E}$ as
$$\pi_i(\underline{E})=\colim_n \pi_{i+n} (E_n).$$

%It is obvious from the definition that a spectrum could have homotopy groups in negative dimensions. If it does not have nontrivial homotopy groups in negative dimensions we call it \textit{$(-1)$-connected} or \textit{connective}.
%\begin{example}\hspace*{15cm}\\
%$\pi_i(\underline{H}G)=G$ if $i=0$ and trivial otherwise, so the Eilenberg-MacLane spectrum is connective.
%\end{example}
\begin{defin}
A \textit{map of spectra} $f: \underline{E} \rightarrow \underline{F}$ is a collection of maps $$f_n: E_n \rightarrow F_n, n \geq 0$$
that commute with the structure maps in $\underline{E}=\{E_n\}$ and $\underline{F}=\{F_n\}$.
\end{defin}
Taking spectra as objects and maps of spectra as morphisms we can define the \textit{category of spectra}. It is denoted by $\text{Spectra}$.

A spectrum can be smashed with a pointed space.
\begin{defin}
Let $\underline{E}=\{E_n\}$ be a spectrum and $X$ be a based space. The spectrum $\underline{E} \wedge X$ is defined by $$(\underline{E} \wedge X)_n=E_n \wedge X.$$
\end{defin}
%Using Example \ref{Ex:Smash} we have $$\Sigma(E_n \wedge X)=(\Sigma E_n) \wedge X \rightarrow E_{n+1} \wedge X,$$ 
Since $\Sigma(E_n \wedge X)\cong(\Sigma E_n) \wedge X$, the structure maps in the spectrum $\underline{E} \wedge X$  are the products of structure maps in $\underline{E}$ and the identity map.
For a spectrum $\underline{E} \wedge X$ the homotopy groups $\pi_i(\underline{E}\wedge X)$ are the groups $\colim_n \pi_{i+n} (E_n \wedge X)$. These groups define a generalized reduced homology theory, determined by $E$.

The following result is a consequence of Proposition 4F.2 in \cite{Hatcher}. See also \cite{Wh:GeneralizedHomology} for more details on representing generalized homology theories with spectra.
\begin{prop}
For the Eilenberg-MacLane spectrum $\underline{H} \Z$ there exists an isomorphism $$\pi_i(X \wedge \underline{H} \Z) \cong \widetilde{\Ho}_i(X;\Z).$$
\end{prop}
If a spectrum $\underline{E}=\{E_n\}_{n \geq 0}$ is an $\Omega$-spectrum, then $\pi_n(\underline{E})$ is
 \begin{equation*}
    \pi_n(\underline{E})=
    \begin{cases}
     \hspace{0.23cm} \pi_n(E_0), & \text{for}\ n \geq 0 \\
      \pi_0(E_{-n}), & \text{for} \ n \leq 0
    \end{cases}
  \end{equation*}
  
Let us note that smash product with a spectrum can be extended from pointed to unpointed spaces.
\begin{defin}
Let $E$ be a spectrum and $X$ an unpointed space. Define the smash product of $E$ and $X$ to be the homotopy fiber of the map
$$E\wedge X_+ \to E$$
induced by the canonical map $X_+\to S^0$.
\end{defin}
For any choice of basepoint in $X$, there is a canonical equivalence between the new and the old definition $E\wedge X$. But the new definition does not depend on a choice of basepoint. This is a variant of the fact that reduced homology can be defined as relative homology to a basepoint, but also can be defined independently of basepoint, using the augmented chain complex.

However, it is also important to note that without a choice of basepoint in $X$, there is no natural map $X\to \Omega^\infty\HZ\wedge X$ representing the Hurewicz homomorphism. Such a map is defined only with a choice of basepoint.

We can assume that each spectrum is an $\Omega$-spectrum up to weak equivalence. Precisely, the following result holds.
\begin{prop}\label{P:equivalentOmega}
Every spectrum is weakly equivalent to an $\Omega$-spectrum.
\end{prop}
If two spectra $\underline{E}$ and $\underline{F}$ are weak equivalent, we write $\underline{E} \simeq \underline{F}$.

Operation $\Sigma^{\infty}$ which assigns to a based space $X$ its suspension spectrum $\Sigma^{\infty}X$, defined by $E_n=\Sigma^n X$ with identities as structure maps, is a functor $$\Sigma^{\infty}: \Top_* \to \text{Spectra}.$$
Its adjoint functor $$\Omega^{\infty}: \text{Spectra} \to \Top_*$$
is defined to be the functor which takes a spectrum $\underline{E}=\{E_n\}_{n \geq 0}$, then replaces it by an equivalent $\Omega$-spectrum $\underline{F}=\{F_n\}_{n \geq 0}$ (which exists using proposition \ref{P:equivalentOmega}) and finally picks off the first place $F_0$. In short, $\Omega^{\infty}(\underline{E})=F_0$ where $\underline{F}=\{F_n\}_{n \geq 0} \simeq \underline{E}$. This $F_0$ is an infinite loop space, which explains the notation.

It follows from the results and comments above that $n^{\text{th}}$ homotopy group of a spectrum $\underline{E}$ equals the $n^{\text{th}}$ homotopy group of the space $\Omega^{\infty}(\underline{E})$.

Finally, let us mention that in addition to the smash product of a spectrum with a space, there is a very important notion of smash product of spectra. For our purposes, the most naive version of the construction suffices. Given two spectra $\underline{E}=\{E_n\}$ and $\underline{F}=\{F_n\}$, we define their smash product $\underline{E}\wedge \underline{F}$ by the formulas $(\underline{E}\wedge \underline{F})_{2n}=E_n\wedge F_n$, and $(\underline{E}\wedge \underline{F})_{2n+1}=E_{n+1}\wedge F_n$, with the structure maps being induced from the structure maps in $\underline{E}$ and $\underline{F}$ in the obvious way. The sphere spectrum is the unit (up to homotopy) for this smash product. 

One feature of smash product of spectra that plays a role in this paper is that unlike smash product of spaces, smash product with a fixed spectrum commutes with finite homotopy limits of spectra. More generally, it commutes with homotopy limits over a category whose classifying space is compact. This is discussed in some detail in~\cite{LRV:HolimsSmash}. The significance for us is that if $F$ is a good presheaf of spectra on $M$, and $\underline{E}$ is a fixed spectrum, then there are natural equivalences %\GregNote{I think there is a lemma to this effect in my paper with Pascal and Ismar}.
\[
\underline{E}\wedge T_kF\simeq T_k\underline{E}\wedge F
\]
and
\[
\underline{E} \wedge L_kF \simeq L_k \underline{E} \wedge F.
\] 
\end{subsection}
\end{section}

\vspace*{0.25cm}
\begin{section}{The homological Taylor tower for reduced $r$-immersions in $\mathbb R^n$}\label{S:main}
The main goal of this paper is to give a convergence result about the homological Taylor tower for the space of $r$-immersions of a smooth manifold $M$ in $\R^n$. As is often the case, when studying the homological tower, it is convenient to replace the functor of $r$-immersions by $r$-immersions ``modulo immersions''. This enables us to express the layers in the Taylor tower in terms of $r$-configuration spaces.

%First we need to introduce the notion of \textit{homological convergence}.
%\vspace*{0.5cm}
%\begin{subsection}{Homotopy of spectra and homology}\label{SS:hsh}

Let $M$ be a smooth manifold. Assume that a basepoint in the space $\Imm(M, \R^n)$ is chosen, and therefore the functor $U\mapsto \Imm(U, \R^n)$ is a presheaf of pointed spaces on $M$. Recall that for $U\subset M$, $\rbar(U, \R^n)$ denotes the homotopy fiber of the map $\rImm(U, \R^n)\to \Imm(U, \R^n)$. Let $\HZ$ denote the Eilenberg-Mac Lane spectrum. The functor 
\[
X\mapsto \HZ \wedge X
\]
represents reduced homology, in the sense that there is a natural isomorphism
\begin{equation}\label{E:homotopy=homology}
\pi_*(\HZ\wedge X )\cong \widetilde{\Ho}_*(X;\Z).
\end{equation}
Furthermore, recall that the functor can be extended to unpointed spaces, by defining $\HZ \wedge X$ for unpointed $X$ to be the homotopy fiber of the map $\HZ \wedge X_+ \to \HZ$.
In this paper we study the following functor
\[
\begin{array}{cccc}
\HZ\wedge \overline{\rImm}(-,\R^n) \colon & \calO(M) & \to &\text{Spectra} \\[5pt]
& U & \mapsto & \HZ \wedge \overline{\rImm}(U,\R^n)
\end{array}
\]
This functor is representing the homology of $\rbar(-, \R^n)$.
\begin{rem}\label{R:Quillen}
Instead of using spectra and the functor $\HZ\wedge -$ to represent homology, we could have used chain complexes and the singular chains functor. One reason for choosing spectra is their topological nature. The category of spectra, and of $\HZ$-module spectra, is tensored and cotensored over topological spaces, while the category of chain complexes is not. Of course, this is a minor technical issue that can be overcome, but anyway it was one reason for us to work with $\HZ$-modules rather than chain complexes. Another reason is that working with $\HZ$-modules readily points to generalizations. In particular, most of our results about the functor $\HZ \wedge \rbar(-, \R^n)$ can be extended to the functor $\Sigma^\infty \rbar(-, \R^n)$, which in turn can be used to obtain information about the unstable Taylor tower of $\rbar(-, \R^n)$.
\end{rem}
\begin{rem}
In \cite{GKW:EmbDisjSurg} and \cite{W:HomEmb}, Goodwillie, Weiss and Klein point out that for a contravariant functor $F: \mathcal{O}(M) \rightarrow \Top$, the cofunctor $\lambda_{\underline{J}}F$ given by $$U \mapsto F(U)_+ \wedge \underline{J}$$ for a fixed spectrum $\underline{J}$ is only "half-good", even if $F$ is good. Namely, it is an isotopy functor but it fails to be finitary. As mentioned in Section \ref{SS:manifoldcalculus}, to fix this they suggest to use the {\it taming} of $\lambda_{\underline{J}}F$. We will denote the taming of a functor such as $\lambda_{\underline{J}}F$ by $\lambda_{\underline{J}}F^\#$. The functor $\lambda_{\underline{J}}F^\#$ is a good cofunctor, and there is a natural transformation $\lambda_{\underline{J}}F\to \lambda_{\underline{J}}F^\#$, which is an equivalence when evaluated on a tame subset of $M$, where by a tame subset we mean an open subset which is diffeomorphic to the interior of a compact manifold with boundary. From now on, whenever we write $\HZ\wedge \overline{\rImm}(-,\R^n)$ we really mean the taming of this functor. In practice it makes no difference since we only are interested in evaluating our functors on tame manifolds.
\end{rem}

So we need to figure out the connectivity of the $k^{\text{th}}$ layer of the Taylor tower for the space $\HZ\wedge \overline{\rImm}(M,\R^n)$. 
By Proposition~\ref{P:ConnOfLayer}, this is determined by the homotopy fiber of the cubical diagram, indexed by subsets of $\{1, \ldots, k\}$,
\begin{equation*}
S\mapsto \HZ \wedge \rbar\left(\coprod_{\underline{k} \setminus S} D^m, \R^n\right).
\end{equation*}
There is a natural map $\rbar(\coprod_{\underline{k} \setminus S} D^m, \R^n)\to \rConf(\underline{k} \setminus S, \R^n)$, which is the composition of the natural map into $\rImm(\coprod_{\underline{k} \setminus S} D^m, \R^n)$, followed by evaluation at the centers of the discs. By the main result of~\cite{AS:rimmrconf}, this map is an equivalence. It follows that the connectivity of the layers of $\HZ\wedge \overline{\rImm}(M,\R^n)$ is determined by the connectivity of the total fiber of the cubical diagram
\begin{equation*}
S\mapsto \HZ \wedge \rConf(\underline{k} \setminus S, \R^n).
\end{equation*}
To analyze the total fiber of this cube, we need to review some facts about the homology of $r$-configuration spaces. This will be done in the next section.

\end{section}

\begin{section}{$r$-configuration spaces in $\mathbb R^n$ as complements of subspace arrangements}\label{S:GM}
We saw in the previous section that the convergence of the Taylor tower of the functor $\HZ \wedge \rbar(-, \R^n)$ is determined by the homology of $r$-configuration spaces $\rConf(k, \R^n)$. These configuration spaces can be interpreted as the complement of an arrangement of subspaces of $(\R^n)^k$. The combinatorics and topology (in particular, homology and cohomology) of subspace arrangements and their complements are well studied. 
Some of main references are \cite{Orlik1980}, \cite{Goresky1980}, \cite{Goresky1983}, \cite{GM1983}, \cite{Bjorner1990}. In particular, the (co)homology of $r$-configuration was studied from this perspective first by Bj\"orner and Welker in~\cite{Bjo95}, and by a number of people after that. In this section we review a qualitative description of the cohomology of $r$-configuration spaces, based on the Goresky-MacPherson formula. We will also describe the effect on cohomology of restriction maps between configuration spaces.

Recall that an $r$-configuration space of $k$ points in $\R^n$ is defined to be the space
$$
\rConf(k, \R^n)=\{(v_1,...,v_k) \in (\R^n)^k : \nexists 1\le i_1<\cdots < i_r \le k \text{ such that } v_{i_1}=...=v_{i_r} \}.
$$
%If $r=2$, this is just the configuration space of ordered $k$-tuples of pairwise distinct points in $X$.
%It can be identified with the space $X^k$ with all the diagonals removed. 
%(we remove all $2$-diagonals $\{(x_1,...,x_k) \in (\R^n)^k : x_{i_1}=x_{i_2} \text{ for some } i_1, i_2 \in \underline{k}\}$, all $3$-diagonals $\{(x_1,...,x_k) \in (\R^n)^k : x_{j_1}=x_{j_2}=x_{j_3} \text{ for some } j_1, j_2, j_3 \in \underline{k}\}$, etc., ending with the thin diagonal $\Delta=\{(x_1,...,x_k) \in \^k : x_{1}=...=x_{k} \}$).
%For $r>2$, $\rConf(k,X)$ is the space $SX^k$ with some, but not all the diagonals removed. 
%In general, $\rConf(k, X)=X^k-\{r$-diagonals, $(r+1)$-diagonals,$ ..., k$-diagonal\}. 
The space $\rConf(k, \R^n)$ is an example of the complement of a subspace arrangement. Let us now recall some formal definitions. %From now on we will only consider the case $X=\R^n$.
%$$\mathcal{A}=\{A^{1}_r,...,A^{t_r}_r, A^{1}_{r+1},...,A^{t_{r+1}}_{r+1}, ..., A^{1}_k,...,A^{t_k}_k\}$$
%so that $A^{j}_i$ is the $i$-diagonal in $(\R^n)^k$ for each $i=r,...,k$, $j=1,...,t_i$.\\
%Number $t_i$ of $i$-diagonals, for $i=r,...,k$, is not obvious. It is $t_k=1$ because there is only one thin diagonal.\\
\begin{defin}
Suppose $I$ is an $r$-tuple of integers $I=(i_1, \ldots, i_r)$, where $1\le i_1<\cdots<i_r\le k$. Let us denote the set of all such $r$-tuples by $k\choose r$. Define
\[
A_I=\{(v_1, \ldots, v_k)\in (\R^n)^k\mid v_{i_1}=\cdots=v_{i_r} \}.
\]
Let $\mathcal{A}=\left\{A_I\mid I\in {k\choose r}\right\}$. When we need to make the set $k$ explicit, we write $\mathcal{A}_k$. More generally, for any set $T$ define $\mathcal{A}_T$ to be the set of ``$r$-equal'' diagonals in $(\R^n)^T$. 
\end{defin}
%$$\calA=\{A_i : A_i \text{ is (at least } r)\text{-equal diagonal for } i=1,...,t\}$$
%in $(\R^n)^k$, where an "(at least $r$)-equal diagonal" is an $l$-diagonal with $l \geq r$. The space $(\R^n)^k$ is isomorphic to the space $\R^{nk}$, so when we say $(\R^n)^k$  here we have in mind $\R^{nk}$. The number $t$ of these diagonals is not obvious.
%It is obvious that this arrangement is central.

%Define now two spaces associated with $\mathcal{A}$, its union and complement.
%The union of subspaces in $\mathcal{A}$ is defined to be
%$$V_{\mathcal{A}}= \bigcup_{i=1}^{t} A_i,$$
%and the complement is $M_{\mathcal{A}}=(\R^n)^k-V_{\mathcal{A}}$. Such defined $M_{\mathcal{A}}$ is an $nk$-dimensional, hence $0$-codimensional, manifold in $(\R^n)^k$. Notice that $M_{\mathcal{A}}$ is just $\rConf(k,\R^n)$.
Note that one can identify $\rConf(k, \R^n)$ with the complement of the union of the $A_I$s:
\[
\rConf(k, \R^n)=(\R^n)^k \setminus \bigcup_{I\in {k\choose r}} A_I.
\]

\begin{example}\hspace*{10cm}
\begin{itemize}
\item If $k<r$, $\rConf(k,\R^n) \cong (\R^n)^k \simeq *$

\item If $k=r$, $\rConf(k,\R^n) \cong (\R^n)^r - \Delta \simeq S^{(r-1)n-1}$, where $\Delta$ is the thin diagonal in $ (\R^n)^r$ and $S^{(r-1)n-1}$ is the sphere of dimension $(r-1)n-1$.

\end{itemize}
\end{example}

%\cite{Grunbaum1971}, \cite{Arnold1969},  \cite{Deligne1972}, \cite{Bri73}, \cite{Cohen1973}, \cite{Cohen:Homology}, \cite{Cohen1978},  \cite{Goresky1988}, \cite{Bjorner1992}, \cite{Cohen1995}, \cite{FZ00}, \cite{LS:CohArrang}, \cite{Gaiffi03}, \cite{Bjorner1994}, \cite{Cohen1989}

%The symmetric group $\Sigma_{k}$ acts freely on $\Conf(k,X)$ by permutting coordinates. Denote the quotient of $\Conf(k,X)$ by this action with $\Conf(k,X)_{\Sigma_k}$. This is the space of unordered configurations of $k$ points in $X$. The cohomology of $\Conf(k,\R^2)_{\Sigma_k}$ was first computed by Fuchs \cite{Fuchs70} and Cohen \cite{Cohen1988}. B{\"o}digheimer, Cohen and Taylor \cite{Cohen1989} have computed the homology groups of $\Conf(k,X)_{\Sigma_k}$ for any odd-dimensional manifold $X$. Arnold \cite{Arnold1969} computed cohomology of $\Conf(k,\R^2)$.
%Here we will express integer cohomologies of partial configuration spaces $\rConf(k,\R^n)$ from the combinatorial data of the complements of some properly taken arrangements, using the work of Goresky and MacPherson. This heavily relies on the work made in \cite{Bjo95}.%Begin with recalling the classic definition. 
%\begin{defin}
%An \textit{(affine) subspace arrangement} $\mathcal{A}=\{A_1,...,A_t\}$ in $\R^m$ is a finite collection of affine proper subspaces $A_i$ of $\R^m$. It is called \textit{central} if all $A_i$, $i=1,...,t$, are linear subspaces, i.e. $0_{\R^m} \in A_i$ for all $i=1,...,t$.\\
%If a subspace $A_i$ of $\R^m$ has the dimension $d$, we say that its \textit{codimension} is $m-d$.
%\end{defin}

The collection $\mathcal{A}$ of linear subspaces of $\R^{nk}$ is an example of a {\it subspace arrangement}. Recall that the {\it intersection lattice} of $\mathcal{A}$ is the poset $\mathcal{L}_{\mathcal A}$ consisting of all the intersections $A_{I_1}\cap\cdots\cap A_{I_t}$ of elements of $\mathcal{A}$, ordered by reverse inclusion. We include in $\mathcal{L}_{\mathcal A}$ the ``empty intersection'' of $A_I$s, which is $\R^{nk}$. The space $\R^{nk}$ is the minimal element of $\mathcal{L}_{\mathcal A}$. It will be denoted by $\hat 0$. The maximal element of $\mathcal{L}_{\mathcal A}$ is the intersection of all the $A_I$, which, assuming $k\ge r$, is the diagonal copy of $\R^n$ in $\R^{nk}$. We denote the maximal element of $\mathcal{L}_{\mathcal A}$ by $\hat 1$.

The poset $\mathcal{L}_{\mathcal A}$ is isomorphic to the poset $\Pi_{k, r}$ of partitions of $\{1, \ldots, k\}$ whose every block is either a singleton or contains at least $r$ elements. We call elements of $\Pi_{k,r}$ $r$-equal partitions of $\{1, \ldots, k\}$. The partitions are ordered from finer to coarser. The isomorphism $\Pi_{k, r}\to \mathcal{L}_{\mathcal A}$ sends a partition $\lambda$ of $\{1, \ldots, k\}$ to the space of $k$-tuples of vectors $(v_1, \ldots, v_k)\in (\R^n)^k$ with the property that $v_i=v_j$ whenever $i$ and $j$ are in the same block of $\lambda$. Equivalently, one can say that $\lambda$ is sent to the space of functions from $\underline{k}$ to $\R^n$ that are constant on each block of $\lambda$. From now on we will identify the posets $\mathcal{L}_{\mathcal A}$ and $\Pi_{k, r}$.

%Recall the classic notions of \textit{intersection semilattice} and \textit{ordered complex}.
%\begin{defin}
%For every affine arrangement $\calA=\{A_1,...,A_t\}$ in $\R^m$, the \textit{intersection semilattice} $L_{\calA}$ is defined to be the collection of all non-empty intersections $A_{i_1} \cap ... \cap A_{i_j}$, $1 \leq i_1 < ...<i_j \leq t$ ordered by reverse inclusion: $x \leq y$ if and only if $y \subseteq x$.
%\end{defin}
%The bottom element of $L_{\calA}$ is $\hat{0}=\R^m$. If this arrangement is central, there is also a top element $\hat{1}=0_{\R^m}$ and $L_{\calA}$ becomes a lattice.

Because $L_{\calA}$ is a partially ordered set, we can define the open interval $(x,y)$ in $L_{\calA}$ to be the set
$$(x,y)=\{z \in L_{\calA} \ \ | \ \ x<z<y\}.$$
\begin{defin}
The \textit{order complex} $\Delta(x,y)$ of an open interval $(x,y)$ in $L_{\calA}$, is the abstract simplicial complex whose vertices are the elements of $(x,y)$ and whose $p$-simplices are the chains $x_0<...<x_p$ in $(x,y)$.
\end{defin}
Let $\widetilde{\Ho}_{i} \Delta(x,y)$ denote the $i^{\text{th}}$ reduced simplicial homology group of $\Delta(x,y)$ with integer coefficients. Similarly, $\widetilde{\Ho}^{i}\Delta(x,y)$ denotes the $i^{\text{th}}$ reduced cohomology group of $\Delta(x,y)$.

The (reduced) cohomology groups of the space $\rConf(k, \R^n)=\R^{nk}\setminus \bigcup_{I\in {k\choose r}} A_I$ can be described in terms of (reduced) homology groups of the order complex of intervals in the intersection lattice of $\calA$. This is known as the Goresky-MacPherson formula. For the original proof of the Goresky-MacPherson formula by means of stratified Morse theory see \cite[Part III]{Goresky1988}. An elementary proof was given by Ziegler and \v Zivaljevi\'c in \cite{ZZ1993}. For the original calculation of the cohomology $\rConf(k, \R^n)$ using the Goresky-MacPherson formula see~\cite{Bjo95}. Here is the statement, in the case relevant to us.
\begin{thm}[Special case of Goresky-MacPherson formula]\label{T:GMFormula}
There is an isomorphism
\begin{equation}\label{E:GMFormula}
\widetilde{\Ho}^{i}(\rConf(k, \R^n)) \cong \bigoplus_{x\in{L^{>\hat{0}}_{\calA}}} \widetilde{\Ho}_{\codim(x)-2-i}\Delta(\hat{0},x)
\end{equation}
\end{thm}
Here, the direct sum is indexed by all $x\neq \hat{0}$ in $L_{\calA}$, and $\codim(x)$ is the codimension of the space $x$ as the subspace of $\R^{nk}$.

For each diagonal $x\in {\mathcal L}_{\mathcal A}$, let $c(x)$ denote the number of components of the partition of $\underline{k}$ which determines the diagonal $x$. Obviously, dimension of $x$ in $(\R^n)^k$ is $\dim(x)=n\cdot c(x)$, so 
\begin{equation}\label{eq: codimension}
\codim(x)=n(k-c(x)).
\end{equation}
%\begin{example}\label{Ex:Conf2/3}
%To illustrate the application of formula \eqref{E:GMConf}, we compute the cohomology groups of $3$-configuration spaces of $4$ and $5$ points using the formula \eqref{E:GMConf}. The intersection lattices that correspond to (partial) configuration spaces of more points are much larger and the cohomology of such spaces is much more difficult to calculate exactly from lattice information.
The following easy example of $3$-configuration spaces of $4$ points illustrates the application of formula \eqref{E:GMFormula}.

\begin{example} Let $\calA$ is the set of all $\text{(at least } 3) \text{-diagonals in } (\R^n)^4$. Then $3\Conf(4,\R^n)=(\R^n)^4-\calA$. The intersection lattice $L_{\calA}$ of $\calA$ is pictured in Figure \ref{fig:3Conf4}. %, where the point $\hat{0}=(\R^n)^4$ corresponds to the partition $\{\{1\}, \{2\}, \{3\}, \{4\}\}$ of the set $\underline{4}$, the point $a$ is the codimension-$2n$ $3$-diagonal in $(\R^n)^4$ which corresponds to the partition $\{\{1,2,3\},\{4\}\}$ of the set $\underline{4}$, the point $b$ corresponds to the partition $\{\{1,2,4\},\{3\}\}$ of $\underline{4}$, the point $c$ corresponds to the partition $\{\{1,3,4\},\{2\}\}$ of $\underline{4}$, the point $d$ corresponds to the partition $\{\{2,3,4\},\{1\}\}$ of $\underline{4}$, the point $e$ is the $3n$-codimensional $4$-diagonal which corresponds to the partition $\{\{1,2,3,4\}\}$ of $\underline{4}$, and $\hat{1}=0_{(\R^n)^4}$ is the origin in $(\R^n)^4$.
Using Theorem~\ref{T:GMFormula}, we find that for every $n> 1$, $$\Ho^{0}(3\Conf(4,\R^n))\cong \Z,$$ $$\Ho^{2n-1}(3\Conf(4,\R^n))\cong \Z^4,$$ $$\Ho^{3n-2}(3\Conf(4,\R^n))\cong \Z^3,$$ and other cohomology groups are trivial.
For $n=1$, the formula is still valid, except that in this case $2n-1=3n-2=1$, so the two cohomology groups add together. So $\Ho^{0}(3\Conf(4,\R))\cong \Z$ and $\Ho^{1}(3\Conf(4,\R))\cong \Z^7.$ For $n=1, 2$, the cohomology of $3\Conf(4,\R^n)$ can be read off the tables at the end of~\cite{Bjo95}.
%Therefore, if $n=1$ we have two nontrivial cohomology groups: $$\Ho^{0}(3\Conf(4,\R))\cong \Z$$ and $$\Ho^{1}(3\Conf(4,\R))\cong \Z^7,$$ and if $n>1$ we have three nontrivial cohomology groups. It is \textbf{easy to show???} that they are the only non-trivial cohomology groups.

%See Figure \ref{fig:3Conf4}.
\end{example}

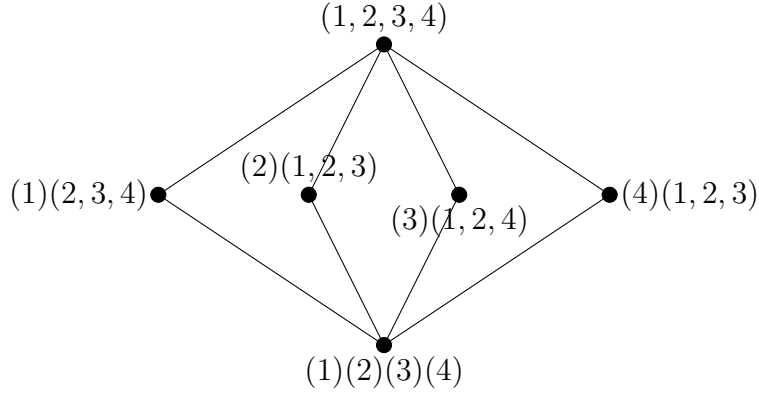
\begin{figure}
\begin{center}
\begin{tikzpicture}
\draw (0,0) node[below] {$(1)(2)(3)(4)$} -- (-3,2) node[left] {$(1)(2,3,4)$}-- (0,4) node[above] {$(1,2,3,4)$};%-- (0,6) node[above] {$\hat{1}$}; 
\draw (0,0)--(-1,2) node[above] {$(2)(1,2,3)$}--(0,4); \draw (0,0)--(1,2) node[below] {$(3)(1,2,4)$}--(0,4);\draw (0,0)--(3,2) node[right] {$(4)(1,2,3)$}--(0,4); \draw[fill] (0,0) circle [radius=0.1]; \draw[fill] (-3,2) circle [radius=0.1]; \draw[fill] (-1,2) circle [radius=0.1]; \draw[fill] (1,2) circle [radius=0.1]; \draw[fill] (3,2) circle [radius=0.1]; \draw[fill] (0,4) circle [radius=0.1]; %\draw[fill] (0,6) circle [radius=0.1];
\end{tikzpicture}
\end{center}
\caption{Intersection lattice for $3\Conf(4,\R^n)$, also known as $\Pi_{4,3}$}
\label{fig:3Conf4}
\end{figure}

For the purpose of analysing the layers in the homological Taylor tower for $r$-immersions it also is desirable to know the effect of restriction maps between $r$-configuration spaces on cohomology. Suppose we have a subset $T\subset \{1, \ldots, k\}$. Then we have a restriction map $\rConf(\underline{k}, \R^n)\to \rConf(T, \R^n)$. We want to describe the induced homomorphism on cohomology, in terms of formula~\eqref{E:GMFormula}. The inclusion $T\hookrightarrow \{1, \ldots, k\}$ induces an inclusion of the poset of $r$-equal partitions of $T$ into the poset of $r$-equal partitions of $\{1, \ldots, k\}$, by making each element of $\{1, \ldots, k\}\setminus T$ into a singleton. Notice that for every $r$-equal partition of $T$, the codimension of the corresponding diagonal is the same whether it is considered a diagonal in $(\R^n)^T$ or in $(\R^n)^k$. This is so because the codimension of a diagonal determined by a partition is determined by the difference between the cardinality of the set and the number of blocks of the partition, by formula~\eqref{eq: codimension}. This number {\it remains unchanged} if one adds some singletons to a partition. Thus we have a homomorphism
\begin{equation}\label{eq: projection}
\bigoplus_{x\in{L^{>\hat{0}}_{{\calA}_T}}} \widetilde{\Ho}_{\codim(x)-2-i}\Delta(\hat{0},x)\to \bigoplus_{x\in{L^{>\hat{0}}_{{\calA}_k}}} \widetilde{\Ho}_{\codim(x)-2-i}\Delta(\hat{0},x)
\end{equation}
which is defined by the inclusion $L^{>\hat{0}}_{{\calA}_T}\hookrightarrow L^{>\hat{0}}_{{\calA}_k}$, and uses the fact that for every $x\in {L^{>\hat{0}}_{{\calA}_T}}$, the number $\codim(x)$ is the same whether $x$ is considered an element of $L^{>\hat{0}}_{{\calA}_T}$ or of $L^{>\hat{0}}_{{\calA}_k}$.
\begin{lemma}\label{lemma: naturality}
The homomorphism $\widetilde{\Ho}^{i}(\rConf(T, \R^n))\to \widetilde{\Ho}^{i}(\rConf(k, \R^n))$ corresponds, under the isomorphism~\eqref{E:GMFormula}, to the homomorphism~\eqref{eq: projection} that we just described.
\end{lemma}
\begin{proof}
This follows easily from the fact that the Goresky-MacPherson formula is natural with respect to inclusions of subarrangements~\cite[Corollary 2.1]{Hu:Arrangements}
\end{proof}
\end{section}

\section{Total fiber of a retractive cubical diagram}\label{S:Retractive}
In general homotopy groups do not commute with total homotopy fibers of cubical diagrams. In this section we will show that for a class of cubes that we call {\it retractive} they do commute. More precisely, we show that for retractive cubes, the homotopy groups of the total fiber are canonically isomorphic to the total kernel of the cube of homotopy groups.

%We assume the reader is familiar with the standard notions of the kernel and cokernel of a homomorphism. The total kernel of a cube is defined to be the iterated kernel of the cube, the same way total fibers are iterated fibers. Let the basepoint in the spectra $\Omega^\infty(\rConf(k-|S|, \R^n) \wedge \HZ)$ corresponds to $0$-class in homology.\\
%\begin{prop}\label{P:commute}There exists an isomorphism
%$$\pi_i(\tfiber(S\longmapsto \Omega^\infty(\rConf(k \setminus S, \R^n)\wedge\HZ))) \cong \tkernel(S\longmapsto \pi_i(\Omega^\infty(\rConf(k-|S|, \R^n)\wedge\HZ)))$$
%\end{prop}
%The proof of Proposition \ref{P:commute} follows from the following two lemmas. First we need to define a \textit{retractive cube} of spectra.

Suppose we have a two-dimensional cubical diagram of spaces or spectra
\begin{equation}\label{E:SquareSpectra}
\xymatrix{
\underline{E_{\emptyset}} \ar[r]^{i_{\emptyset, 1}}\ar[d]_{i_{\emptyset, 2}} & \underline{E_1}\ar[d]^{i_{1, 12}} \\
\underline{E_2} \ar[r]_{i_{2, 12}} & \underline{E_{12}}
}
\end{equation}
Suppose that all the maps in the square \eqref{E:SquareSpectra} have homotopy sections, so that the square of sections
$$
\xymatrix{
\underline{E_{12}} \ar[r]^{s_{12, 1}}\ar[d]_{s_{12,2}} & \underline{E_1}\ar[d]^{s_{1, \emptyset}} \\
\underline{E_2} \ar[r]_{s_{2, \emptyset}} & \underline{E_{0}}
}
$$
commutes up to homotopy, and so that the following mixed square
$$
\xymatrix{
\underline{E_{2}} \ar[r]^{i_{2, 12}}\ar[d]_{s_{2, \emptyset}} & \underline{E_{12}}\ar[d]^{s_{12, 1}} \\
\underline{E_0} \ar[r]_{i_{\emptyset, 1}} & \underline{E_{1}}
}
$$
also commutes up to homotopy. Note that the vertical maps in the mixed square are sections, while the horizontal maps are from the original square.

Let us call a square \eqref{E:SquareSpectra} with such sections a \textit{retractive square}.

More generally, let us define a retractive cubical diagram as follows. 
\begin{defin}
Let $\chi$ be a $k$-dimensional cubical diagram. We say that $\chi$ is \textit{retractive} if for every $U\subset \{1, \ldots, k\}$ and every $i\notin U$, the map $\chi(U)\to \chi(U\cup \{i\})$ has a homotopy section, the cube of sections commutes up to homotopy, and furthermore whenever $U\subset \{1, \ldots, k\}$, and $i, j\in \{1, \ldots, k\}\setminus U$, with $i<j$, the following mixed square commutes up to homotopy
 $$
\xymatrix{
\chi(U\cup \{j\}) \ar[r]\ar[d] & \chi(U\cup \{i, j\}) \ar[d] \\
\chi(U) \ar[r] & \chi(U\cup\{i\})
}.
$$
\end{defin}

\begin{lemma}\label{L:commute}
Let $\chi$ be a retractive $k$-dimensional cubical diagram of spectra. Let $E_*$ be any homology theory, and let $E^*$ be a cohomology theory. Then $E_*(\tfiber\chi)$ (resp. $E^*(\tfiber \chi)$) is a direct summand of $E_*(\chi(\emptyset))$ (resp. of $E^*(\chi(\emptyset))$). Moreover, the following natural homomorphism is an isomorphism:
\[
E_*(\tfiber\chi)\xrightarrow{\cong} \tkernel\left(E_*\chi\right).
\]
Similarly, there is a natural isomorphism
\[
\tcokernel(E^*\chi)\xrightarrow{\cong} E^*(\tfiber\chi).
\]
\end{lemma}
\begin{proof}
We will prove the claim for homology. The proof of the cohomological statement is the same, reversing all arrows. The proof is by induction on $k$, starting with with the case $k=1$, which is elementary and well-known. Let us review it anyway. A retractive $1$-dimensional cube is a map $\chi(\emptyset) \to \chi(1)$, together with a homotopy section $\chi(1)\to \chi(\emptyset)$. The total fiber of the cube is the homotopy fiber of the map $\chi(\emptyset) \to \chi(1)$. By homotopy section we mean that the composition $\chi(1) \to \chi(\emptyset)\to \chi(1)$ is a weak equivalence. It follows that the composition $E_*\chi(1) \to E_*\chi(\emptyset)\to E_*\chi(1)$ is an isomorphism. From here it readily follows that the long exact sequence in $E_*$ associated with the fibration sequence $\tfiber\chi\to \chi(\emptyset)\to \chi(1)$ splits as a direct sum of split short exact sequences in each degree. Furthermore it readily follows that the following homomorphisms are isomorphisms
\[
E_*\tfiber\chi\xrightarrow{\cong} \ker\left(E_*\chi(\emptyset)\to E_*\chi(1)\right) \xrightarrow{\cong} \coker\left(E_*\chi(1)\to E_*(\chi(\emptyset))\right).
\]

Now suppose the lemma holds for cubes of dimension less than $k$ and let $\chi$ be a retractive cube of dimension $k$. Let $\chi_1$ and $\chi_2$ be $k-1$-dimensional cubes defined by $\chi_1(U)=\chi(U)$ and $\chi_2(U)=\chi(U\cup \{k\})$. Then $\chi$ can be identified with the natural map of cubes $\chi_1\to \chi_2$. The cubes $\chi_1$ and $\chi_2$ are retractive, so by induction hypothesis, the lemma holds for them. The retractions do not quite define a map of cubes $\chi_2\to \chi_1$, because we only assumed that the mixed squares commute up to homotopy. But they do define a homomorphism of cubes $E_*\chi_2\to E_*\chi_1$, which is a section of the homomorphism of cubes $E_*\chi_1\to E_*\chi_2$. We have the following diagram

\[
\xymatrix{
 {E_*\tfiber \chi}	\ar[r] & {E_*\tfiber\chi_1} \ar[r]\ar[d]^\cong & {E_*\tfiber\chi_2} \ar[d]^\cong \\
{\tkernel E_*\chi_2} \ar[r] \ar@/_1pc/[rr]_\cong& {\tkernel E_*\chi_1} \ar[r] & {\tkernel E_*\chi_2}
}
\]

%\[\begin{tikzcd}
% {E_*\tfiber \chi}	& {E_*\tfiber\chi_1} & {E_*\tfiber\chi_2} \\
%	{\tkernel E_*\chi_2} & {\tkernel E_*\chi_1} & {\tkernel E_*\chi_2}
%	\arrow[from=1-1, to=1-2]
%	\arrow[from=1-2, to=1-3]
%	\arrow["\cong", from=1-2, to=2-2]
%	\arrow[from=2-2, to=2-3]
%	\arrow[from=2-1, to=2-2]
%	\arrow["\cong"', bend right=15, from=2-1, to=2-3]
%	\arrow["\cong"', from=1-3, to=2-3]
%\end{tikzcd}\]
The top row is induced by applying $E_*$ to a fibration sequence of spectra. The vertical homomorphisms are isomorphisms by induction hypothesis. It follows that the upper right homomorphism is a split surjection, and the top row is a split short exact sequence in each dimension. Furthermore, $E_*\tfiber\chi$ maps isomorphically onto the kernel of the bottom right map, which is $\tkernel E_*\chi$.
\end{proof}
\section{The cube of $r$-configuration spaces is retractive}\label{S: Conf is retractive}
\begin{lemma}\label{lemma: config retractive}
The $k$-cube of spaces $$S \mapsto \rConf(\underline{k} \setminus S, \R^n) $$ is retractive for $n\ge 2$.
\end{lemma}
\begin{proof}
Let $T$ be a finite set and suppose $x\notin T$. Our first step it to construct a section to the restriction map
\[
r_{T\cup\{x\}, T}\colon \rConf(T\cup\{x\}, \R^n)\to \rConf(T, \R^n).
\]
Let $p_1\colon \R^n \to \R$ be projection onto the first coordinate. Define a map
\[
s_{T, T\cup\{x\}}\colon \rConf(T, \R^n) \to \rConf(T\cup\{x\}, \R^n)
\]
as follows. An element of $\rConf(T, \R^n)$ is a function $f\colon T\to \R^n$ with the property that no $r$ points of $T$ go to the same point. Extend $f$ to a function from $T\cup\{x\}$ by sending $x$ to
\[
(\max \{p_1f(t)\mid t\in T\}+1, 0, \ldots, 0).
\]
In words, $x$ is sent to the point of $\R^n$ whose first  coordinate is one more than the maximal first coordinate of the existing points, and all other coordinates are zero. It is clear that the image of $x$ is different from all the other points in the configuration. Thus if $f$ was an $r$-immersion, then the resulting map $T\cup\{x\}\to \R^n$ is still an $r$-immersion. We have defined a map $s_{T, T\cup\{x\}}\colon\rConf(T, \R^n) \to \rConf(T\cup\{x\}, \R^n)$.
It is clear that the following composition is the identity (not even just homotopic to the identity but is the actual identity map)
\[
\rConf(T, \R^n) \xrightarrow{s_{T, T\cup\{x\}}} \rConf(T\cup\{x\}, \R^n)\xrightarrow{r_{T\cup\{x\}, T}}\rConf(T, \R^n).
\]
It follows that $s_{T, T\cup\{x\}}$ is a section of $r_{T\cup\{x\}, T}$. Next, we need to show that whenever $x, y\notin T$, the following diagram commutes up to homotopy
$$
\begin{aligned}
\xymatrix{
\rConf(T,\R^n) \ar[r] \ar[d] & \rConf(T\cup \{x\},\R^n) \ar[d] \\
\rConf(T\cup\{y\},\R^n) \ar[r]       &  \rConf(T\cup\{x,y\},\R^n)
}
\end{aligned}
$$
It is for this step that we need to assume $n\ge 2$. Let $f\colon T\to \R^n$ represent an element of $\rConf(T,\R^n)$. The images of $f$ in $\rConf(T\cup\{x,y\},\R^n)$ under the two ways around the diagram are two extensions of $f$ from $T$ to $T\cup\{x, y\}$. One of the extensions sends $x$ to $(\max \{p_1f(t)\mid t\in T\}+1, 0, \ldots, 0)$, and sends $y$ to $(\max \{p_1f(t)\mid t\in T\}+2, 0, \ldots, 0)$. The other extension does the same thing, with $x$ and $y$ switched. It is clear that one can write a homotopy between the two maps, by swapping the images of $x$ and $y$ along a circle in the plane spanned by the first two coordinates of $\R^n$.

Finally we need to check that the following mixed square commutes up to homotopy
$$
\begin{aligned}
\xymatrix{
\rConf(T \cup\{x\},\R^n) \ar[r] \ar[d] & \rConf(T,\R^n) \ar[d] \\
\rConf(T\cup\{x, y\},\R^n) \ar[r]       &  \rConf(T\cup\{y\},\R^n)
}
\end{aligned}.
$$
This, too, is clear. In fact, it is easy to check that there is a well-defined straight line homotopy between the two maps around the square.

We have shown that the section maps that we have defined make the cube of $r$-configuration spaces and restriction maps between them into a retractive cube.
\end{proof}

\begin{section}{Connectivity of the cube of (co)homologies of $r$-configuration spaces}\label{S:ccrcs}
We have seen that the cube of spaces $S\mapsto \rConf(\underline{k} \setminus S, \R^n)$, where $S$ ranges over the subsets of $\{1, \ldots, k\}$ is retractive (Lemma~\ref{lemma: config retractive}). It follows that the cube of spectra obtained by applying the suspension spectrum functor to it, i.e., the cube
\begin{equation}\label{eq: suspension}
S\mapsto \Sigma^\infty\rConf(\underline{k} \setminus S, \R^n),
\end{equation}
is also retractive. 

Our goal is to analyse how cartesian is the cube $S\mapsto \HZ\wedge \Sigma^\infty \rConf(\underline{k} \setminus S, \R^n)$. Smash product commutes with total fibers of cubical diagrams of spectra. Therefore, the answer is the same as for the cubical diagram~\eqref{eq: suspension}. However, we want to use the description of the {\it co}homology of $r$-configuration spaces given by the Goresky-MacPherson formula. The following lemma says that the homology and cohomology groups of the relevant spectrum are isomorphic.
\begin{lemma}
The homology and cohomology groups of the total fiber of~\eqref{eq: suspension} are (non-canonically) isomorphic.
\end{lemma}
\begin{proof}
It is known, for example by the results of~\cite{Bjo95}, that the homology groups of the space $\rConf(k, \R^n)$, and therefore also of the suspension spectrum of this space, are finitely generated free abelian groups. Since the cube $\Sigma^\infty\rConf(\underline{k} \setminus S, \R^n)$ is retractive, it follows by Lemma~\ref{L:commute} that the homology of the total fiber of the cube $\Sigma^\infty \rConf(\underline{k} \setminus S, \R^n)$ is a direct summand of the homology of $\Sigma^\infty \rConf(k, \R^n)$. Therefore, the homology groups of the total fiber are also finitely generated free abelian groups. Therefore they are isomorphic to the cohomology groups of the total fiber, by the universal coefficients theorem.
\end{proof}
It follows that the homological connectivity of the total fiber of~\eqref{eq: suspension} is equivalent to the cohomological connectivity. Next, we give a qualitative description of the cohomology of the total fiber, in the style of Theorem~\ref{T:GMFormula}.

Let $\Pi_{\geq r}(\underline{k})$ denote the set partitions of $\underline{k}$ with the property that each component has at least $r$ elements (i.e., elements of $\Pi_{k, r}$ without singletons). 
\begin{lemma}\label{L:totalcokernel}
The $i$-th cohomology group of the total fiber of the cube~\eqref{eq: suspension} is isomorphic to the following direct sum:
\begin{equation}\label{E:tcokernel}
\bigoplus_{x\in \Pi_{\geq r}(\underline{k})} \widetilde{\Ho}_{\codim{(x)}-2-i}\Delta(\hat{0},x)
\end{equation}
\end{lemma}
\begin{proof}
The cube~\eqref{eq: suspension} is retractive. Using the cohomological part of Lemma~\ref{L:commute}, we conclude that the $i$-th cohomology of the total fiber is isomorphic to the cokernel of the homomorphism
\[
\bigoplus_{i=1}^k \widetilde{\Ho}^i\rConf(k \setminus \{i\}, \R^n) \to \widetilde{\Ho}^i\rConf(k, \R^n).
\]
By Lemma~\ref{lemma: naturality}, this homomorphism can be identified with the following homomorphism
\begin{equation}
\bigoplus_{i=1}^k\bigoplus_{x\in{L^{>\hat{0}}_{{\calA}_{\underline{k} \setminus\{i\}}}}} \widetilde{\Ho}_{\codim(x)-2-i}\Delta(\hat{0},x)\to \bigoplus_{x\in{L^{>\hat{0}}_{{\calA}_k}}} \widetilde{\Ho}_{\codim(x)-2-i}\Delta(\hat{0},x)
\end{equation}
The homomorphism maps each summand in the source isomorphically onto a summand in the target (some summands in the source go to the same summand in the target, so the homomorphism is not injective). The image of the homomorphism is the sum of terms corresponding to $r$-equal partitions with at least one singleton. The cokernel is the direct sum of terms corresponding to $r$-equal partitions that do not have a singleton.
\end{proof}

%\end{section}
%\vspace*{0.5cm}
%\begin{section}{Connectivity of the cube of homologies of $r$-configuration spaces}\label{S:CokernelConnectivity}

It follows from Lemma~\ref{L:totalcokernel} that to find how cartesian the cube \eqref{eq: suspension} is, we need to find the smallest $i$ for which the homology group
\begin{equation}\label{E:homologyofcomplex}
\widetilde{\Ho}_{\codim(x)-2-i}\Delta(\hat{0},x)
\end{equation}
is non-trivial for some $x\in \Pi_{\ge r}(\underline{k})$.  

Throughout this section, let $x$ be a partition of $\{1, \ldots, k\}$ where each block has at least $r$ elements. Recall that $c(x)$ denotes the number of blocks of $x$. Note that if $k_1, \ldots, k_{c(x)}$ are the sizes of the blocks of $x$, then $k_1+\cdots+k_{c(x)}=k$. Let $[\hat 0, x]$ be the closed interval in $\Pi_{k, r}$.
\begin{lemma}
Let $x$ be as above. Suppose $x$ has $c(x)$ blocks, of sizes $k_1, \ldots, k_{c(x)}$. Then there is an isomorphism of posets
\[
[\hat 0, x]\cong \Pi_{k_1, r}\times \cdots\times \Pi_{k_{c(x)}, r}.
\]
\end{lemma}
\begin{proof}
The interval $[\hat 0, x]$ consists of $r$-equal partitions of $\{1, \ldots, k\}$ that are refinements of $x$. This is the same data as an $r$-equal partition of each block of $x$, which is the same as an element of $\Pi_{k_1, r}\times \cdots\times \Pi_{k_{c(x)}, r}$.
\end{proof}
Given a poset $\calP$ with a minimum and maximum element, let $\calP^0$ be the poset $\calP$ with the minimum and maximum removed.
\begin{cor}
Let $x$ be as in the previous lemma. Then there is a homotopy equivalence ($*$ denotes join)
\[
|\Delta(\hat 0, x)|\simeq \Sigma^{c(x)-1} |\Pi_{k_1, r}^0| *\cdots* |\Pi_{k_{c(x)}, r}^0|.
\]
\end{cor}
\begin{proof}
This follows from the lemma, and the well-known fact that given two posets $\calP$ and $\calQ$ with minimum and maximum objects, there is a homotopy equivalence~\cite[Theorem 5.1 (d)]{Walker:Posets} 
$$|(\calP \times \calQ)^0| \simeq \Sigma |\calP^0|*|\calQ^0|.$$
\end{proof}

\begin{lemma}\label{L:dimensionofcomplex}
Let $x$ be as in the previous lemma and corollary. Then $|\Delta(\hat{0},x)|$ is homotopy equivalent to a complex of dimension $k-c(x)(r-1)-2.$
Furthermore, the homology of $|\Delta(\hat{0},x)|$ in dimension $k-c(x)(r-1)-2$ is non zero.
\end{lemma}
\begin{proof}
By the corollary, the space $|\Delta(\hat{0},x)|$ is homotopy equivalent to $\Sigma^{c(x)-1} |\Pi_{k_1, r}^0| *\cdots* |\Pi_{k_{c(x)}, r}^0|$. By the results of~\cite{Bjo95}, $|\Pi_{k, r}^0|$ is homotopy equivalent to a wedge of spheres, not all of the same dimension, and the top homology of this space occurs in dimension $k-r-1$. It follows that the space $\Sigma^{c(x)-1} |\Pi_{k_1, r}^0| *\cdots* |\Pi_{k_{c(x)}, r}^0|$ is a wedge of spheres, with the top homology occurring in dimension $$c(x)-1 + (k_1-r-1)+\cdots+(k_{c(x)}-r-1) + c(x)-1= k - c(x)(r-1)-2.$$
\end{proof}

\begin{example}
If $r \leq k <2r$, there is only one summand $x$ in \eqref{E:tcokernel} - this is the partition $\{\underline{k}\}$, or in other words the thin diagonal. For this $x$, $\dim \Delta (\hat{0},x)=k-r-1$. 
\end{example}
Now we can state and prove the main result of this section
\begin{prop}\label{prop: conf connectivity}
When $r\le n+1$, the cube~\eqref{eq: suspension} is $k(n-1)+\left\lfloor \frac{k}{r}\right\rfloor (r-n-1)$-cartesian.

When $r\ge n+1$, the cube~\eqref{eq: suspension} is $k(n-1)+r-n-1$-cartesian.
\end{prop}
\begin{rem}\label{remark: borderline}
Note that when $r=n+1$ both formulas say that the cube~\eqref{eq: suspension} is $k(n-1)$-cartesian.
\end{rem}
\begin{proof}
Given $x$, the smallest $i$ for which the homology \eqref{E:homologyofcomplex} might be non-trivial is one that satisfies $\codim(x)-2-i=\dim \Delta(\hat{0},x)$. Using Lemma \ref{L:dimensionofcomplex} we have that the smallest $i$ for which the total cokernel \eqref{E:tcokernel} might be non-trivial is one that satisfies
$$
\codim(x)-2-i=k-c(x)(r-1)-2.
$$

Because $\codim(x)=n(k-c(x))$ for $x \in \Pi_{\geq r}(\underline{k})$, it follows that

\begin{equation}\label{E:LowestHomology}
i=k(n-1)+c(x)(r-n-1).
\end{equation}

We have to see for which $x$ this number $i$ is the smallest possible. We distinguish between two overlapping cases, depending on the sign of $r-n-1$.\\

1) When $r-n-1\le 0$, i.e. when $r\le n+1$, finding $i$ as small as possible is the same as finding $x \in \Pi_{\geq r}(\underline{k})$ with the biggest number $c(x)$ of components. Since all components have to be of the size at least $r$, the largest number of them is attained when there is a maximum number of them of the size $r$. In that case, $c(x)=\left\lfloor \frac{k}{r}\right\rfloor$, so the smallest $i$ is $$i=k(n-1)+\left\lfloor \frac{k}{r}\right\rfloor (r-n-1).$$ So in this case, the cubical diagram \eqref{eq: suspension} is $k(n-1)+\left\lfloor \frac{k}{r}\right\rfloor (r-n-1)$-cartesian.

%2) When $r-n-1=0$, i.e. $r=n+1$, then from \eqref{E:LowestHomology} $$i=k(n-1),$$ so the connectivity of the total cokernel \eqref{E:tcokernel} is $k(n-1)-1$.

2) When  $r-n-1\ge 0$, i.e. $r\ge n+1$, finding $i$ as small as possible is the same as finding $x \in \Pi_{\geq r}(\underline{k})$ with the smallest number $c(x)$ of components. Thus we need $c(x)$ to be equal to $1$. This $x$ is actually the thin diagonal in the space $(\R^n)^k$ that corresponds to the partition $\{\underline{k}\}$ of $\underline{k}$. In that case,
$$i=k(n-1)+r-n-1,$$
hence~\eqref{eq: suspension} is $k(n-1)+r-n-1$-cartesian.
\end{proof}
%It is easy to test these results on examples from Section \ref{S:GM}.

%Note that $r\leq k$ implies $1\leq k/r$ and $k\leq 2r-1$ implies $k/r <2$. Therefore, $\left\lfloor \frac{k}{r}\right\rfloor=1$ for $r \leq k \leq 2r-1$, and $k(n-1)+\lfloor \frac{k}{r}\rfloor (r-n-1)=k(n-1)+r-n-1$. So, for $r \leq k \leq 2r-1$ the smallest $i$'s from the cases 1) and 3) are equal.\\

\end{section}
\vspace*{0.5cm}
\begin{section}{Convergence result}\label{convergenceresult}
%Let $M$ be a smooth manifold. For the sake of simplicity, define the \textit{homological connectivity} of the map 
%\begin{equation}\label{E:mapTk}
%T_k\rImm(M,\R^n)\longrightarrow T_{k-1}\rImm(M,\R^n),
%\end{equation}
%to be the connectivity of the map
%$$
%T_k(\Omega^{\infty}(\rImm(M,\R^n) \wedge \HZ)) \rightarrow T_{k-1}(\Omega^{\infty}(\rImm(M,\R^n) \wedge \HZ))
%$$
%between two successive stages of the "homological" Taylor tower for $\rImm(M,\R^n)$, i.e. of the Taylor tower for
%$$
%\Omega^{\infty}(\rImm(M,\R^n) \wedge \HZ).
%$$

Let $M$ be a smooth manifold of dimension $m$. Now we finally can calculate the connectivity of the map
\begin{equation}\label{E:mapTk}
T_k \HZ\wedge \rbar(M,\R^n)\to T_{k-1}\HZ\wedge \rbar(M,\R^n).
\end{equation}

Knowing that $c_k$-connectivity of the total fiber of the cube \eqref{eq: suspension} implies $(c_k-km+1)$-connectivity of the map \eqref{E:mapTk}, we can find the conditions under which the Taylor tower converges, using results from Section \ref{S:ccrcs}. There are three different cases.\\

1) For $r-n-1<0$, the connectivity of the map \eqref{E:mapTk} is
\begin{equation}\label{E:Conn:r-n-1<0}
\begin{aligned}
k(n-1)+\left\lfloor \frac{k}{r}\right\rfloor(r-n-1)-1-mk+1 &= k(n-m-1)+\left\lfloor \frac{k}{r}\right\rfloor(r-n-1)\\
&=  k(n-m-1)+ \left(\frac{k}{r}-\frac{k\hspace{-8pt}\mod r}{r}\right) (r-n-1)\\
&= k\left(n-m-\frac{n}{r}-\frac{1}{r}\right)-\frac{k\hspace{-8pt}\mod r}{r} (r-n-1)\\
&= k\left(n\frac{r-1}{r}-m-\frac{1}{r}\right)-\frac{k\hspace{-8pt}\mod r}{r} (r-n-1)
\end{aligned}
\end{equation}
where we noted that $\left\lfloor \frac{k}{r}\right\rfloor = k/r - (k\hspace{-4pt}\mod r)/r$.
%Then \eqref{E:LayerConnr-n-1<0} becomes
Note now that
$$
-\frac{k\hspace{-8pt}\mod r}{r} (r-n-1)
$$
is nonnegative since $r-n-1<0$.
This means that, as long as
$$
n\frac{r-1}{r}-m-\frac{1}{r} >0,
$$
the connectivities increase with $k$.\\

2)  For $r-n-1=0$, the connectivity of the map \eqref{E:mapTk} is
\begin{equation}\label{E:Conn:r-n-1=0}
k(n-1)-1-mk+1=k(n-m-1),
\end{equation}
which goes to $+\infty$ as $k\longrightarrow +\infty$ if $n-m-1>0$.\\

3) For $r-n-1>0$, the connectivity of the map \eqref{E:mapTk} is
\begin{equation}\label{E:Conn:r-n-1>0}
k(n-1)+r-n-2-mk+1=k(n-m-1)+r-n-1,
\end{equation}
which goes to $+\infty$ as $k\longrightarrow +\infty$ if $n-m-1>0$.\\

Thus we proved the following theorem.
\begin{thm} (Homological convergence of the Taylor tower for $r$-immersions in $\mathbb R^n$)\label{T:Convergence}\\
Let $M$ be an $m$-dimensional smooth manifold and $\R^n$ the $n$-dimensional Euclidean space. Assume $n>1$. Let $\rImm(M,\R^n)$ be the space of $r$-immersions of $M$ in $\R^n$. Consider the map
\begin{equation*}
p_k\colon T_k\HZ\wedge \rbar(M, \R^n)\to T_{k-1}\HZ\wedge \rbar(M, \R^n).
\end{equation*}
\\
a) For $r\le n+1$ the map $p_k$ is $$k\left(n\frac{r-1}{r}-m-\frac{1}{r}\right)-\frac{k\hspace{-8pt}\mod r}{r} (r-n-1)$$-connected. The tower converges intrinsically if $n>\frac{rm+1}{r-1}.$

b)  For $r \ge n+1$ the map $p_k$ is $k(n-m-1)+r-n-1$-connected. The tower converges intrinsically if $n> m+1$.
%$$
%b) Suppose $r=n+1$. Then the homological connectivity of the map $\mathcal{M}$ is $$k(n-m-1)$$ and $\mathcal{T}$ converges if $$
%n>m+1.
%$$
%b)  Suppose $r \ge n+1$. Then the map $p_k$ is $k(n-m-1)+r-n-1$-connected, and the tower converges intrinsically if $$
%n>m+1.
%$$
\end{thm}
\begin{proof}
Only the assertions regarding intrinsic convergence remain to be checked. The tower converges intrinsically if the connectivity of $p_k$ approaches $\infty$ with $k$. In the case $r\le n+1$, since $(k\!\!\mod r)$ is a bounded function of $k$, this is equivalent to the condition $n\frac{r-1}{r}-m-\frac{1}{r}>0$, which is the same as $n>\frac{rm+1}{r-1}$. In the case $r\ge n+1$, the formula for the connectivity of $p_k$ clearly tells us that the connectivity goes to $\infty$ if $n>m+1$. \end{proof}
%\begin{equation}\label{E:HomologyConvergenceCondition}
%n\frac{r-1}{r}-m-\frac{1}{r} >0,
%\end{equation}
%the numbers increase with $k$ and the Taylor tower converges.\\

\end{section}
\vspace*{0.5cm}
\begin{section}{Comparing with the unstable tower}\label{section:compare}
In this section we will compare the layers, and the connectivities of the maps in the Taylor tower of $\HZ\wedge \rbar(M, \R^n)$ with those in the Taylor tower of the unstabilized functor $\rImm(M, \R^n)$. We will show that roughly up to degree $2r-1$ the connectivities of the maps in the two towers are the same, and the first non-trivial homotopy groups of the layers are isomorphic. %In the process we extend some of the calculations done in~\cite{SSV:r-imm}. 

In this section, let us assume that we chose a basepoint in $\rImm(M, \R^n)$ rather than just in $\Imm(M, \R^n)$, so that the presheaf $\rbar(-, \R^n)$ takes values in pointed spaces. We have a diagram of presheaves
\begin{equation}\label{eq: expanded comparison}
\rImm(-, \R^n)\xleftarrow{i} \rbar(-, \R^n) \xrightarrow{s}  \Omega^\infty \Sigma^\infty \rbar(-, \R^n) 
\xrightarrow{h} \Omega^\infty\HZ\wedge \rbar(-, \R^n).
\end{equation}
The map $i$ induces an equivalence of derivatives and layers beyond the first one. The map $h$ induces the Hurewicz homomorphism on each layer and each stage of the Taylor tower. In particular, it induces the Hurewicz isomorphism on the first non-trivial homotopy group of each layer. In Theorem~\ref{thm: comparison} we address the question for which $k$ the map $s$, and therefore also  $h\circ s$, induces an isomorphism on the first nontrivial homotopy group of the $k$-th layer. When $r=2$, the answer is known to be: only for $k=2$. We show that for $r>2$ the answer is: for all $k\le 2r-1$, with a small caveat for $r=3$.

\begin{thm}\label{thm: comparison}
Assume $0<\dim(M)<n$, $r>2$.

For $1< k < r$, the following maps are equivalences:

\[T_k\rImm(M, \R^n)\to T_1\rImm(M, \R^n)\simeq \Imm(M, \R^n)\]

and

\[T_k \HZ\wedge \rbar(M, \R^n)\to T_1 \HZ\wedge \rbar(M, \R^n) \simeq *.\]

For $r\le k \le 2r-1$, the connectivity of the map $T_k\rImm(M, \R^n)\to T_{k-1}\rImm(M, \R^n)$ is the same as the connectivity of the map $T_k\HZ\wedge \rbar(M, \R^n)\to T_{k-1}\HZ\wedge \rbar(M, \R^n)$. Note that the connectivity of these maps,  and therefore the connectivity of the fibers of these maps, can be read off Theorem~\ref{T:Convergence}.

For $r>3$ and $r\le k\le 2r-1$, and also for $r=3$ and $r\le k < 2r-1$, the maps $s$ and $h$ induce isomorphisms between the first non-trivial homotopy groups of the $k$-th layers of the functors $\rbar(-, \R^n)$, $\Omega^\infty \Sigma^\infty \rbar(-, \R^n)$, and $\Omega^\infty\HZ\wedge \rbar(-, \R^n)$. 

When $r=3$, $k=2r-1=5$, the map $s$, and therefore also $h\circ s$ in diagram~\eqref{eq: expanded comparison}, induces an epimorphism on the first non-trivial homotopy group of the $k$-th layer. 
\end{thm}
\begin{rem}
The case $k=r, r+1$ of the last assertion of the theorem can be obtained by comparing our Theorem~\ref{T:Convergence} with the calculations done in~\cite{SSV:r-imm}.
\end{rem}
\begin{proof}
The assertion that $T_1\rImm(M, \R^n)\simeq \Imm(M, \R^n)$ follows from the fact that when $M=D^m$, the following maps are  equivalences~\cite{AS:rimmrconf}
\[
\Emb(D^m, \R^n)\xrightarrow{\simeq} \rImm(D^m, \R^n) \xrightarrow{\simeq} \Imm(D^m, \R^n),
\]
together with the fact that the functor $\Imm(-, \R^n)$ is linear, at least on manifolds whose handle dimension is less than $n$.

The assertion that the towers for $\rImm(-, \R^n)$ and for $\Sigma^\infty \rbar(-, \R^n)$ are both constant for $k<r$ follows from the fact that the derivatives of both functors vanish below degree $r$. Indeed, the $k$-th layer in the Taylor tower of $\rImm(M, \R^n)$ is determined by the following $k$-dimensional cubical diagram
\[
S\mapsto \rImm(\coprod_{\underline{k}\setminus S} D^m, \R^m).
\] 
By the result of~\cite{AS:rimmrconf}, this cubical diagram is equivalent to the diagram
\[
S\mapsto \mathrm{L}(\R^m, \R^n)^{\underline{k}\setminus S}\times \rConf({\underline{k}\setminus S}, \R^n).
\] 
Here $\mathrm{L}(\R^m, \R^n)$ is the space of injective linear maps from $\R^m$ to $\R^n$. This is the ``tangential data'' of an immersion. When $k>1$ the tangential data cancels out, and the last cube is as cartesian as the following cube
\begin{equation}\label{eq:unstablecube}
S\mapsto \rConf({\underline{k}\setminus S}, \R^n).
\end{equation}
On the other hand, the $k$-th layer in the Taylor tower of $ \Sigma^\infty \rbar(M, \R^n)$ is determined by the following $k$-dimensional cubical diagram
\begin{equation}\label{eq:stablecube}
S\mapsto \Omega^\infty\Sigma^\infty \rConf({\underline{k}\setminus S}, \R^n).
\end{equation}
When $k<r$, and $S\subseteq \underline{k}$ the space $\rConf({\underline{k}\setminus S}, \R^n)$ is contractible. Therefore for $1<k<r$ the cubes~\eqref{eq:unstablecube} and~\eqref{eq:stablecube} are cubes of contractible spaces. These cubes are homotopy cartesian for trivial reasons, and therefore the maps $T_k\rImm(M, \R^n)\to T_{k-1}\rImm(M, \R^n)$ and  $T_k \HZ\wedge \rbar(M, \R^n)\to T_{k-1} \HZ\wedge \rbar(M, \R^n)$ are equivalences for $1<k<r$. This proves that both towers are constant in the range $1\le k \le r-1$.

Now let us suppose that $r\le k \le 2r-1$. To prove the assertion about the connectivities of the maps in the two towers, we need to show that the cubical diagram~\eqref{eq:unstablecube} is as cartesian as the diagram~\eqref{eq:stablecube} in the indicated case. Furthermore, we want to prove that the map $s$ in~\eqref{eq: expanded comparison} induces an isomorphism/epimorphism on the first non-trivial homotopy groups of the total homotopy fibers in the appropriate cases.

The map $s$ induces the following map of cubical diagrams, indexed by the poset of subsets $S\subset \underline{k}$,
\begin{equation}\label{eq:mapofcubes}
\rConf({\underline{k}\setminus S}, \R^n)\xrightarrow{s} \Omega^\infty\Sigma^\infty \rConf({\underline{k}\setminus S}, \R^n).
\end{equation}
The spaces $\rConf({\underline{k}\setminus S}, \R^n)$ are $(r-1)n-2$-connected. By Freudenthal suspension theorem, the maps~\eqref{eq:mapofcubes} are $2(r-1)n-3$-connected. On the other hand, both cubes are retractive cubes by Lemma~\ref{lemma: config retractive}. It follows that the homotopy groups of the total homotopy fibers of both cubes are isomorphic to the total kernels of the corresponding cubes of homotopy groups. Proposition~\ref{prop: conf connectivity} tells us the connectivity of the total homotopy fiber of~\eqref{eq:stablecube}, and therefore also the connectivity of the total kernel of the corresponding cube of homotopy groups. If this connectivity is smaller than (resp. equals to) the connectivity of the maps in~\eqref{eq:mapofcubes}, then~\eqref{eq:mapofcubes} induces an isomorphism (resp: an epimorphism) between the first non-trivial homotopy groups of the total homotopy fibers. So we have to check that the range provided by Proposition~\ref{prop: conf connectivity} is smaller than (or equals to) $2(r-1)n-3$ in the cases indicated in the statement that we are trying to prove.

Suppose first that $r> n+1$. In this case, Proposition~\ref{prop: conf connectivity} says that~\eqref{eq:stablecube} is $k(n-1)+r-n-1$-cartesian. So we have to check that the inequality
\[
k(n-1)+r-n-1< 2(r-1)n-3
\]
holds whenever $k<2r$. Simplifying, we obtain the inequality
\[
k<\frac{(2n-1)r-3}{n-1} - 1.
\]
So it is enough to check the inequality
\[
2r\le \frac{(2n-1)r-3}{n-1} - 1.
\]
Multiplying by $n-1$ we obtain the inequality
\[
2r(n-1)\le (2n-1)r-3-n+1,
\]
%Multiplying out, we obtain
%\[
%2rn-2r\le 2nr-r-n-2
%\]
which is equivalent to $r\ge n+2$, which is what we assumed.

Now suppose that $r\le n+1$. Then Proposition~\ref{prop: conf connectivity} says that~\eqref{eq:stablecube} is $k(n-1)+\left\lfloor\frac{k}{r}\right\rfloor(r-n-1)$-cartesian. So we have to check that the inequality
\[
k(n-1)+\left\lfloor\frac{k}{r}\right\rfloor(r-n-1)< 2(r-1)n-3
\]
holds when $r\le k < 2r$, with the exception that when $r=3$, $k=5$ it is in fact an equality. The reader can check that in this case we do indeed obtain the equality
\[
5(n-1)+\left\lfloor\frac{5}{3}\right\rfloor(3-n-1)=4n-3.
\]
In other cases, the assumption $r\le k < 2r$ implies $\left\lfloor\frac{k}{r}\right\rfloor=1$. 
So we have to check the inequality
\[
k(n-1)+r-n-1< 2(r-1)n-3.
\]
We can rewrite the inequality as follows
\[
k(n-1)< (2r-1)(n-1) + r - 3,
\]
or equivalently
\[
k< 2r-1 + \frac{r-3}{n-1}.
\]
For $r=3$ this inequality is equivalent to $k<5$. For $3<r \le n+1$, this holds for all $k\le 2r-1$, as stated.
\end{proof}
\end{section}

\vspace*{0.5cm}
\begin{section}{Further questions}\label{S:Further}
1. We gave conditions on the $m$ and $n$ that guarantee intrinsic convergence of the Taylor tower of $\HZ\wedge \rbar(M,\R^n)$. The next question is, what does the Taylor tower from Theorem \ref{T:Convergence} converge to? It is natural to guess that whenever the Taylor tower converges intrinsically, it actually converges to $\HZ\wedge \rbar(M,\R^n)$.

2. What can one say about the convergence of the Taylor tower for the unstable functor $\rImm(M,\R^n)$? The question of intrinsic convergence of the unstable tower might be tractable, and is a good place to start. One can use the methods of this paper to describe the layers of the functor $\Sigma^\infty \rbar(M, \R^n)$. Given this, one can try to analyse the layers of the functor $\rbar(M, \R^n)$ via the cobar construction \[\mathrm{cobar}(\Omega^\infty, \Sigma^\infty \Omega^\infty, \Sigma^\infty \rbar(M, \R^n)),\] in the style of~\cite{AC:OperadsChain}. It is conceivable that one can use these methods to obtain conditions on $m, n$, and $r$ that guarantee that the tower converges intrinsically.

Then there is a question of what the tower actually converges to. Once again, it seems reasonable to guess that whenever the Taylor tower of a ``natural'' functor converges intrinsically, then it actually converges to the functor.

3. What can one say about $r$-immersions into a general manifold $N$? In order to understand the layers of the tower of the functor $\Sigma^\infty \rbar(M, N)$ one needs to understand (the stable homotopy type of) the homotopy fiber of the map $\rConf(k, N)\to N^k$. For $r=2$ this homotopy fiber was analysed in~\cite{Arone:DerI}, and it seems likely that a similar analysis can be done for general $r$.

4. Construct interesting invariants/obstructions to existence of $r$-immersions, using the Taylor tower. In this paper we focused on situations where the connectivity of the $k$-th layer in the tower goes to infinity as $k$ goes to infinity. But cases when the connectivity does not go to infinity also can be interesting. Of particular potential interest are situations where the layers are all either $-1$-connected or $-2$-connected. In the former case, the bottom homotopy groups of the layers give invariants, in the latter case they give obstructions to existence.

For example, it follows from Theorem~\ref{T:Convergence} that when $n=m+1$ and $r=n+1$, then all the layers of $\HZ\wedge (n+1)\Imm(M, \R^n)$ are $-1$-connected. The $0$-th homotopy groups of the layers should give invariants of $r$-immersions. In the case $n=2$, and say $M=S^1$, $3\Imm(S^1, \R^2)$ is the space of smooth curves in $\R^2$ that do not have triple intersections. Spaces of such curves were studied quite intensely, starting with Arno{l'}d~\cite{Arnold:PlaneCurves, Tabachnikov:PlaneCurves, Shumakovich:Formulas}. In particular, Arno{l'}d developed the theory of finite type invariants for such curves. We expect these invariants to show up in the Taylor tower of  $\HZ\wedge 3\Imm(S^1, \R^2)$. In particular, we speculate that the first non-trivial layer of the tower, which by Theorem~\ref{thm: comparison} is the third layer, detects the ``Strangeness'' invariant, defined in~\cite{Arnold:PlaneCurves} and studied further in~\cite{Tabachnikov:PlaneCurves} and~\cite{Shumakovich:Formulas}.
\end{section}

\vspace*{0.5cm}

\bibliographystyle{alpha}

%\bibliography{Bibliography}

\pagestyle{plain}

\end{document}